\documentclass[12pt]{amsart}
\usepackage{amsmath}	
\usepackage{amssymb}

\newtheorem{theorem}{Theorem}[section]
\newtheorem{lemma}[theorem]{Lemma}
\newtheorem{proposition}[theorem]{Proposition}
\newtheorem{corollary}[theorem]{Corollary}

\theoremstyle{remark}

\numberwithin{equation}{section}

\newcommand{\et}{\quad\mbox{and}\quad}

\newcommand{\bQ}{\mathbb{Q}}
\newcommand{\bR}{\mathbb{R}}
\newcommand{\bZ}{\mathbb{Z}}
\newcommand{\cO}{{\mathcal{O}}}
\newcommand{\hlambda}{\widehat{\lambda}}
\newcommand{\norm}[1]{\|\/#1\/\|}

\newcommand{\uu}{\mathbf{u}}
\newcommand{\uv}{\mathbf{v}}
\newcommand{\uw}{\mathbf{w}}
\newcommand{\ux}{\mathbf{x}}

\newcommand{\uy}{\mathbf{y}}
\newcommand{\uz}{\mathbf{z}}
\newcommand{\tbigwedge}{\textstyle{\bigwedge}}

\newcommand{\wdeltax}{\widetilde{\Delta\ux}}
\newcommand{\wddeltax}{\widetilde{\Delta^2\ux}}

\begin{document}

\baselineskip=14.7pt

\title[Simultaneous approximation]
{On simultaneous approximation to\\ a real number, its square, and its cube, II}
\author{Damien Roy}
\address{
   D\'epartement de math\'ematiques\\
   Universit\'e d'Ottawa\\
   150 Louis Pasteur\\
   Ottawa, Ontario K1N 6N5, Canada}
\email{droy@uottawa.ca}

\begin{abstract}
In a previous paper with the same title, we gave an upper bound for the exponent
of uniform rational approximation to a quadruple of $\bQ$-linearly independent real 
numbers in geometric progression.  Here, we explain why this upper bound is not 
optimal.  
\end{abstract}

\maketitle

\section{Introduction}
\label{sec:intro}

For each positive integer $n$ and each real number $\xi$, we follow Bugeaud and Laurent
in \cite{BL2005} and denote by $\hlambda_n(\xi)$ the exponent of uniform rational 
approximation to the geometric progression $(1,\xi,\dots,\xi^n)$ of length $n+1$ and 
ratio $\xi$.  This is defined as the supremum of all $\lambda\in \bR$ for which the 
inequalities
\[
 |x_0|\le X, \quad \max_{1\le i\le n} |x_0\xi^i-x_i| \le X^{-\lambda}
\]
admit a non-zero solution $\ux=(x_0,\dots,x_n)\in\bZ^{n+1}$ for each large enough 
real number $X$.   In their 1969 seminal paper \cite{DS1969}, Davenport and Schmidt
established upper bounds for $\hlambda_n(\xi)$ that are independent of $\xi$ when
$[\bQ(\xi):\bQ]>n$, namely when $1,\xi,\dots,\xi^n$ are linearly independent 
over $\bQ$.  Then, using an argument of geometry of numbers, they deduced a result 
of approximation to such $\xi$ by algebraic integers of degree at most $n+1$.  For 
$n=1$ and $n=2$, both estimates are best possible.  For $n=1$, we have 
$\hlambda_1(\xi)=1$ for each $\xi\in \bR\setminus\bQ$, and the corresponding result
of approximation by algebraic integers of degree at most $2$ is best possible as explained
in \cite[\S 1]{DS1969}.  For $n=2$, it is shown in \cite[Theorem 1a]{DS1969} that,
for each $\xi\in\bR$ with $[\bQ(\xi):\bQ]>2$, we have $\hlambda_2(\xi)\le 1/\gamma
\cong 0.618$ where $\gamma=(1+\sqrt{5})/2$ stands for the golden ratio.  In 
\cite{R2004}, we showed that this upper bound is best possible and, in \cite{R2003},
that the corresponding result of approximation by algebraic integers of degree at 
most $3$ is also best possible.  For $n>2$, refined
upper bounds for $\hlambda_n(\xi)$ have been established in \cite{Ba2022,La2003,
PR2022, R2008, Schl2020, Schl2021} but the least upper bound is unknown.  
This paper deals with the case $n=3$.

Let $\lambda_3\cong0.4245$ denote the smallest positive root of $T^2-\gamma^3T+\gamma$
where $\gamma=(1+\sqrt{5})/2$, as above.  In the previous paper \cite{R2008}
with the same title, I proved the following statement.

\begin{theorem}
 \label{intro:thm1}
Let $\xi\in\bR$ with $[\bQ(\xi)\colon\bQ] >3$, and let $c$ and
$\lambda$ be positive real numbers.  Suppose that, for any
sufficiently large value of $X$, the inequalities
\begin{equation}
  \label{intro:thm1:eq}
 |x_0|\le X, \quad
 \max_{1\le i\le 3}|x_0\xi^i-x_i| \le cX^{-\lambda}
\end{equation}
admit a non-zero solution $\ux=(x_0,x_1,x_2,x_3)\in \bZ^4$. Then, we
have $\lambda\le \lambda_3$.  Moreover, if $\lambda=\lambda_3$, then
$c$ is bounded below by a positive constant depending only on $\xi$.
\end{theorem}

In particular, this implies that any $\xi\in\bR$ with $[\bQ(\xi)\colon\bQ] >3$
has $\hlambda_3(\xi)\le \lambda_3$. 

For several years, before the publication of \cite{R2008}, I thought that the 
upper bound $\lambda_3$ for $\lambda$ in Theorem \ref{intro:thm1} could be 
optimal until I realized that it is not.  However, I did not include 
the proof of this as it was only leading to a small improvement over $\lambda_3$.
The goal of this paper is to present that argument in the hope that it will help 
finding the least upper bound.  In fact, we will prove the 
following result.

\begin{theorem}
 \label{intro:thm2}
Under the hypotheses of Theorem \ref{intro:thm1}, we
have $\lambda< \lambda_3$.  
\end{theorem}

Using the same method, it is possible to compute an explicit $\epsilon>0$ 
such that $\hlambda_3(\xi)\le \lambda_3-\epsilon$.  I still refrain to do that here 
in order to keep the presentation as simple as possible.   In a further paper, 
I plan to provide more tools to make progress on this problem.  

In the next two sections, we recall most of the results of \cite{R2008} with
some precision added, including the notion of minimal points and the definition of the
important polynomial map $C\colon(\bR^4)^2\to\bR^2$ that was already implicit 
in \cite{DS1969}.  In Section \ref{sec:Psi}, we introduce a new pair of polynomial maps 
$\Psi_-$ and $\Psi_+$ from $(\bR^4)^3$ to $\bR^4$ and we elaborate on their 
analytic and algebraic properties.   
In Sections \ref{sec:first} to \ref{sec:another} we use these tools to study the behavior 
of the minimal points assuming that the hypothesis of Theorem \ref{intro:thm1} 
holds with $\lambda=\lambda_3$.  In each section, we get new algebraic relations 
that link the minimal points.  In Section \ref{sec:algC}, they involve the polynomial
map $C$ and, in Section \ref{sec:another}, the maps $\Psi_\pm$.  In the process, 
we isolate a very rigid structure among the subspaces spanned by consecutive 
minimal points.  Using this, we end up with a contradiction in Section 
\ref{sec:final}, and this proves Theorem \ref{intro:thm2}.   For some of the main 
results that we establish along the way, we indicate weaker conditions on $\lambda$ 
for which they hold, but we omit the proof to keep the paper reasonably short.  
In an addendum, we provide a further algebraic relation
involving another polynomial map with interesting algebraic properties.

%
%

\section{Notation and preliminaries}
\label{sec:not}

The notation is the same as in \cite{R2008}.  We fix a real number $\xi$ with
 $[\bQ(\xi)\colon\bQ] >3$ and a real number $\lambda>0$ which fulfills the 
hypothesis of Theorem \ref{intro:thm1} for some constant $c>0$.  For
shortness, we use the symbols $\ll$ and $\gg$ to denote inequalities involving
multiplicative constants that depend only on $\xi$ and $\lambda$.  We also 
denote by $\asymp$ their conjunction.  As we are not interested in the dependence
on $c$, we consider that $c\asymp 1$, contrary to what is done in \cite{R2008}.

For each integer $n\ge 1$ and each point $\ux=(x_0,x_1,\dots,x_n)\in\bR^{n+1}$,
we define
\begin{gather}
  \ux^- = (x_0,\dots,x_{n-1}), \quad \ux^+ = (x_1,\dots, x_n), \quad
  \Delta\ux = \ux^+ - \xi\ux^-,
  \label{not:eq1}\\
 \norm{\ux}=\max_{0\le i\le n}|x_i| \et L(\ux)=\max_{1\le i\le n}|x_0\xi^i-x_i|.
  \label{not:eq2}
\end{gather}
For each $p=1,\dots,n+1$, we identify $\tbigwedge^p\bR^{n+1}$ with 
$\bR^{\binom{n+1}{p}}$ via an ordering of the Grassmann coordinates as in
\cite[Chap.~I, \S5]{Sc}.  If $n\in\{1,2,3\}$ and if $\ux$ is a non-zero point of 
$\bZ^{n+1}$, then $L(\ux)\neq 0$ and we have
\begin{equation}
  L(\ux) \asymp\norm{\Delta\ux} \asymp \norm{\ux\wedge(1,\xi,\dots,\xi^n)}.
  \label{not:eq3}
\end{equation}

As in \cite[\S2]{R2008}, we fix a sequence $(\ux_i)_{i\ge 1}$ of non-zero
points of $\bZ^4$ with the following properties:
\begin{itemize}
\item[(a)] the positive integers $X_i:=\norm{\ux_i}$ form a strictly
 increasing sequence;
 \item[(b)] the positive real numbers $L_i:=L(\ux_i)$ form a
 strictly decreasing sequence;
 \item[(c)] if some non-zero point $\ux\in\bZ^4$ satisfies $L(\ux)
 < L_i$ for some $i\ge 1$, then $\|\ux\| \ge  X_{i+1}$.
\end{itemize}
This is slightly different than the construction of Davenport and Schmidt 
in \cite[\S4]{DS1969}, but it plays the same role.  In particular, using 
\eqref{not:eq3}, our hypothesis translates into the basic estimate
\begin{equation}
  L_i \asymp\norm{\Delta\ux_i} \ll X_{i+1}^{-\lambda}.
  \label{not:eq4}
\end{equation}
We say that $(\ux_i)_{i\ge 1}$ is a sequence of \emph{minimal points}
for $\xi$ in $\bZ^4$.

For any integer $n\ge 1$, we define the \emph{height} of a non-zero vector 
subspace $V$ of $\bR^n$ defined over $\bQ$ to be 
\begin{equation*}
  H(V) = \norm{\uy_1\wedge\cdots\wedge\uy_p},
\end{equation*}
where $(\uy_1,\dots,\uy_p)$ is any basis of $V\cap\bZ^n$ over $\bZ$.  We also 
set $H(\{0\})=1$.  We now recall some definitions and results from 
\cite[\S3]{R2008} relative to the subspaces $\langle\ux_i,\dots,\ux_j\rangle_\bR$
of $\bR^4$ spanned by consecutive minimal points $\ux_i,\dots,\ux_j$.  They use 
the well-known fact that, since $[\bQ(\xi):\bQ]>3$, any proper subspace of 
$\bR^4$ contains finitely many minimal points (see \cite[Lemma 2.4]{R2008}).

We first recall that each $\ux_i$ is a primitive point of $\bZ^4$, namely a non-zero 
point of $\bZ^4$ whose gcd of the coordinates 
is $1$.  Thus, $\langle\ux_i\rangle_\bR$ has height $X_i$ for each $i\ge 1$, and so
$\langle\ux_i\rangle_\bR \neq \langle\ux_j\rangle_\bR$ for distinct integers $i,j\ge 1$.
For each $i\ge 2$, we define
\[
  W_i = \langle\ux_{i-1},\ux_i\rangle_\bR.
\]
Then $W_i$ has dimension $2$ and the set $I$ of integers $i\ge 2$ for which 
$W_i\neq W_{i+1}$ is infinite.  For each $i\in I$, we define the 
\emph{successor} of $i$ in $I$ to be the smallest element $j$ of $I$ with $j>i$.
We also say that elements $i<j$ of $I$ are \emph{consecutive} in $I$, if $j$ is the
successor of $i$ in $I$.  For such $i$ and $j$, we have
\[
  W_i \neq W_{i+1}=\cdots=W_j \neq W_{j+1},
\]
thus $W_{i+1}=W_j=\langle\ux_i,\dots,\ux_j\rangle_\bR=
\langle\ux_i,\ux_j\rangle_\bR$.  For each $i\in I$, we also define
\[
  U_i = W_i+W_{i+1} = \langle\ux_{i-1},\ux_i,\ux_{i+1}\rangle_\bR.
\]
Then $U_i$ has dimension $3$.  Finally, we define $J$ to be the set of all $i\in I$
for which $U_i\neq U_j$ where $j$ is the successor of $i$ in $I$.  This is an infinite 
subset of $I$.  For each triple of consecutive elements $h<i<j$ of $I$, we have
\[
  U_i = W_{h+1}+W_{i+1} 
       = \langle\ux_h,\ux_i\rangle_\bR + \langle\ux_i,\ux_j\rangle_\bR
      = \langle\ux_h,\ux_i,\ux_j\rangle_\bR,
\]
thus $(\ux_h,\ux_i,\ux_j)$ is a basis of $U_i$.  We also note that 
$U_i = \langle\ux_h,\dots,\ux_j\rangle_\bR$.  Moreover, we have 
$\ux_{j+1}\notin U_i$ if and only if $i\in J$.  The heights of these subspaces 
of $\bR^4$ can be estimated as follows.

\begin{proposition}
\label{not:prop}
$\mathrm{(i)}$ For each $i\ge 2$, the pair $(\ux_{i-1},\ux_i)$ is a basis of
 $W_i\cap\bZ^4$ and we have 
\[
 H(W_i) \asymp X_iL_{i-1} \ll X_i^{1-\lambda}.
\]
\begin{itemize}
\item[$\mathrm(ii)$] For each $i\in I$, we have 
\[
 X_iH(U_i) \ll H(W_i)H(W_{i+1}).
\]
\item[$\mathrm(iii)$] For each pair of consecutive elements $i<j$ of $I$
 with $i\in J$, we have
\[
 H(W_j) \ll H(U_i)H(U_j).
\]
\end{itemize}
\end{proposition}

Part (i) is \cite[Lemma 3.1]{R2008}.  Part (ii) follows from a general inequality of 
Schmidt from \cite[Chap.~I, Lemma 8A]{Sc} on the basis that 
$W_i$ and $W_{i+1}$ have sum $U_i$ and intersection $\langle\ux_i\rangle_\bR$.  
Part (iii) follows from the same formula upon noting that the sum of $U_i$ and $U_j$ 
is $\bR^4$ with height $1$ and that their intersection is $W_j$.  

Determinants play a crucial role in this theory.  For each integer $n\ge 0$ and each 
choice of $\uy_i=(y_{i,0},\dots,y_{i,n})\in\bR^{n+1}$ for $i=0,\dots,n$, we denote by
$\det(\uy_0,\dots,\uy_n)$ the determinant of the matrix $(y_{i,j})$ whose rows are
$\uy_0,\dots,\uy_n$.  We will need the following formula. 

\begin{lemma}
\label{not:lemma1}
Suppose that $n\ge 1$. Then, for $\uy_0,\dots,\uy_n$ as above, we have
\[
 \det(\uy_0,\dots,\uy_n)
 = \sum_{i=0}^n (-1)^i y_{i,0}
        \det(\Delta\uy_0,\dots,\widehat{\Delta\uy_i},\dots,\Delta\uy_n)
\]
where the hat on $\Delta\uy_i$ in the right hand side indicates that this point is 
omitted from the list.
\end{lemma}

\begin{proof}
The linear map $\varphi\colon\bR^{n+1}\to\bR^{n+1}$ sending each 
$\uy=(y_0,\dots,y_n)\in\bR^{n+1}$ to 
\[
 \varphi(\uy)=(y_0,y_1-\xi y_0,\dots,y_n-\xi y_{n-1})=(y_0,\Delta\uy)
\]
has determinant $1$.  Thus, 
the square matrix with rows $\uy_0,\dots,\uy_n$ has the same determinant
as that with rows $\varphi(\uy_0),\dots,\varphi(\uy_n)$.  The result follows 
by expanding the determinant of this matrix along its first column.
\end{proof}

The formula of Lemma \ref{not:lemma1} yields the standard estimate
\begin{equation}
\label{not:eq5}
  |\det(\uy_0,\dots,\uy_n)|
  \ll \sum_{i=0}^n \norm{\uy_i}L(\uy_0)\cdots\widehat{L(\uy_i)}\cdots L(\uy_n) 
\end{equation}
for any choice of $\uy_0,\dots,\uy_n\in\bR^{n+1}$ with $n\le 3$.  We add the 
condition $n\le 3$ so that the implicit constant in \eqref{not:eq5} is independent
of $n$.  In this paper, we will need finer estimates of the following form.

\begin{corollary}
\label{not:cor}
Let $n\in\{1,2,3\}$ and let $\uy_0,\dots,\uy_n$ be linearly independent elements
of $\bZ^{n+1}$.  We have 
\[
 |\det(\uy_0,\dots,\uy_n)|
 \asymp \norm{\uy_n}\, |\det(\Delta\uy_0,\dots,\Delta\uy_{n-1})|
\]
if $L(\uy_n)<1$ and if the $n$ products $\norm{\uy_i}L(\uy_0)\cdots
\widehat{L(\uy_i)}\cdots L(\uy_n)$ with $i=0,\dots,n-1$ are smaller than some 
positive function $\delta$ of $\xi$.
\end{corollary}

\begin{proof}
Put $d=\det(\uy_0,\dots,\uy_n)$. Lemma \ref{not:lemma1} yields
\[
 \big| d - (-1)^n\uy_{n,0}\det(\Delta\uy_0,\dots,\Delta\uy_{n-1}) \big|
  \le c\sum_{i=0}^{n-1} \norm{\uy_i}L(\uy_0)\cdots\widehat{L(\uy_i)}\cdots L(\uy_n)
\]
for some $c=c(\xi)>0$.  Since $d$ is a non-zero integer, we have $|d|\ge 1$.  So, if
the conditions of the corollary are fulfilled with $\delta=1/(2nc)$, we obtain 
\[
 \big| d - (-1)^n\uy_{n,0}\det(\Delta\uy_0,\dots,\Delta\uy_{n-1}) \big|
  \le 1/2 \le |d|/2,
\]
and the result follows since the condition $L(\uy_n)<1$ implies that $\norm{\uy_n}
\asymp |y_{n,0}|$.
\end{proof}

We also recall that \eqref{not:eq5} generalizes to
\begin{equation}
\label{not:eq6}
  \norm{\uy_0\wedge\cdots\wedge\uy_p}
  \ll \sum_{i=0}^p \norm{\uy_i}L(\uy_0)\cdots\widehat{L(\uy_i)}\cdots L(\uy_p) 
\end{equation}
for any choice of $\uy_0,\dots,\uy_p\in\bR^{n+1}$ with $0\le p\le n\le 3$.
We conclude with the following estimates from \cite[Lemma 2.1]{R2008}.

\begin{lemma}
\label{not:lemma3}
Let $C\in\bZ^2$ and $\ux\in\bZ^{n+1}$ with $n\in\{1,2,3\}$.  Then
$\uy=C^+\ux^--C^-\ux^+\in\bZ^n$ satisfies
\[
 \norm{\uy}\ll\norm{\ux}L(C)+\norm{C}L(\ux) \et L(\uy)\ll \norm{C}L(\ux).
\]
\end{lemma}

%
%

\section{The maps $C$ and $E$}
\label{sec:CE}

For each non-zero point $\ux$ of $\bR^4$, we define
\[
 V(\ux) =\langle \ux^-,\ux^+\rangle_\bR \subseteq \bR^3.
\]
We also define a polynomial map $C\colon\bR^4\times\bR^4\to\bR^2$ by 
\[
 C(\ux,\uy)=\big(\det(\ux^-,\ux^+,\uy^-),\, \det(\ux^-,\ux^+,\uy^+)\big)
\]
and note that, for a given $(\ux,\uy)\in\bR^4\times\bR^4$, we have 
\begin{equation}
\label{CE:eq:Cneq0}
 C(\ux,\uy)\neq 0 \ \Longleftrightarrow \ 
  \big(\, \dim V(\ux)=2 \ \text{and}\ V(\uy)\not\subseteq V(\ux)\, \big).
\end{equation}
Since $C$ is quadratic in its first argument, there is a unique tri-linear map
$E\colon(\bR^4)^3\to\bR^2$ such that 
\begin{equation}
\label{CE:eq:defE}
  E(\uw,\ux,\uy)=E(\ux,\uw,\uy) \et E(\ux,\ux,\uy)=2C(\ux,\uy)
\end{equation}
for each choice of $\uw,\ux,\uy\in\bR^4$.  It is given by
\[
 E(\uw,\ux,\uy)=\big( \det(\uw^-,\ux^+,\uy^-) - \det(\uw^+,\ux^-,\uy^-),\, 
     \det(\uw^-,\ux^+,\uy^+) - \det(\uw^+,\ux^-,\uy^+)\big).
\]
Besides \eqref{CE:eq:defE}, we note that this map satisfies
\begin{equation}
\label{CE:eq:Eyxy}
  E(\ux,\uy,\uy)=E(\uy,\ux,\uy)=-C(\uy,\ux)
\end{equation}
for each $(\ux,\uy)\in\bR^4\times\bR^4$.

The following result uses the operator $\Delta$ defined in \eqref{not:eq1}.  We write
$\Delta^2$ to denote its double iteration.  Thus, for a point $\ux\in\bR^4$,
we have $\Delta^2\ux=\Delta(\Delta\ux)$.  We also denote by $\Delta\ux^-$ 
the vector $\Delta(\ux^-) = (\Delta\ux)^-$, omitting parentheses.  Similarly,
$\Delta\ux^+$ stands for $\Delta(\ux^+) = (\Delta\ux)^+$.

\begin{lemma}
\label{CE:lemmaC}
For any $\ux=(x_0,\dots,x_3)$ and $\uy=(y_0,\dots,y_3)$ in $\bR^4$, we have
\begin{align*}
 C(\ux,\uy)^- &=x_0\det(\Delta^2\ux,\Delta\uy^-)
                        + y_0\det(\Delta\ux^-,\Delta^2\ux) + \cO(L(\ux)^2L(\uy)),\\[2pt]
 C(\ux,\uy)^+ &=x_0\det(\Delta^2\ux,\Delta\uy^+)
                        + y_0\xi\det(\Delta\ux^-,\Delta^2\ux) + \cO(L(\ux)^2L(\uy)),\\[2pt]
 \Delta C(\ux,\uy) &=x_0\det(\Delta^2\ux,\Delta^2\uy) + \cO(L(\ux)^2L(\uy)).
\end{align*}
\end{lemma}

\begin{proof}
For any choice of sign $\epsilon$, we have
\[
 C(\ux,\uy)^\epsilon = \det(\ux^-,\ux^+,\uy^\epsilon) 
  = \det(\ux^-,\Delta\ux,\uy^\epsilon).
\]
Thus, Lemma \ref{not:lemma1} gives
\[
 C(\ux,\uy)^\epsilon  
  = x_0\det(\Delta^2\ux,\Delta\uy^\epsilon) 
     + (\uy^\epsilon)_0\det(\Delta\ux^-,\Delta^2\ux) + \cO(L(\ux)^2L(\uy)),
\]
where $(\uy^-)_0=y_0$ and $(\uy^+)_0=y_1=y_0\xi+\cO(L(\uy))$.  This explains 
the first two formulas.  The last one follows from them by definition of $\Delta$.
\end{proof}

The above estimates have the following immediate consequence.

\begin{corollary}
\label{CE:corC}
For any $\ux,\uy\in\bR^4$, we have
\[
 \norm{C(\ux,\uy)} \ll \norm{\ux}L(\ux)L(\uy) + \norm{\uy}L(\ux)^2
 \et
 L(C(\ux,\uy)) \ll \norm{\ux}L(\ux)L(\uy).
\]
\end{corollary}

For shortness, we write
\begin{equation}
\label{CE:eq:Cij}
 V_i=V(\ux_i) \et C_{i,j}=C(\ux_i,\ux_j)
\end{equation}
for each pair of positive integers $i$ and $j$.  Then we have the following 
non-vanishing result.

\begin{lemma}
\label{CE:lemmaCneq0}
Suppose that $\lambda > \sqrt{2}-1 \cong 0.4142$.  There is an integer $i_0\ge 1$
with the following properties.
\begin{itemize}
\item[\textrm{(i)}]  
  We have $\dim V_i=2$ and $V_i\neq V_{i+1}$ for any integer $i\ge i_0$.
\item[\textrm{(ii)}] 
  For any integer $i\ge i_0$ and any non-zero $\uy\in\bZ^3$, there is a choice 
  of signs $\epsilon$ and $\eta$ for which the integer
  $\det(\ux_i^\epsilon, \ux_{i+1}^\eta, \uy)$ is non-zero.
\item[\textrm{(iii)}] 
  For any pair of consecutive elements $i<j$ of $I$ with $i\ge i_0$, 
  the four points $C_{i,i+1}$, $C_{i,j}$, $C_{j,j-1}$ and $C_{j,i}$ are 
  all non-zero, and $C_{i,j}=bC_{i,i+1}$ for some non-zero integer $b$ 
  with $|b|\asymp X_j/X_{i+1}$.
\end{itemize}
\end{lemma}

\begin{proof}
For each sufficiently large integer $i\ge 1$, we have $\dim V_i=2$ 
by \cite[Lemma 2.3]{R2008} and $V_i\neq V_{i+1}$ by \cite[Proposition 5.2]{R2008}. 
Thus property (i) holds for some integer $i_0\ge 1$.  We now show that (ii) and (iii) also
hold for such $i_0$. 

To prove (ii), fix an integer $i$ with $i\ge i_0$ and a non-zero point $\uy\in\bZ^3$.
If $\uy\in V_i$, we can write $V_i=\langle \ux_i^\epsilon,\uy\rangle_\bR$ for a choice 
of sign $\epsilon$, and then $\bR^3=V_i+V_{i+1}
= \langle \ux_i^\epsilon, \ux_{i+1}^\eta, \uy\rangle_\bR$ for a choice of sign $\eta$.
If $\uy\notin V_i$, then $\langle \ux_i^\epsilon,\uy\rangle_\bR$ is a subspace of $\bR^3$ 
of dimension $2$ for any choice of sign $\epsilon$.  Choosing $\epsilon$ such that
$\ux_i^\epsilon\notin V_{i+1}$, we find again that $\bR^3  
= \langle \ux_i^\epsilon,\uy\rangle_\bR + V_{i+1}
= \langle \ux_i^\epsilon, \ux_{i+1}^\eta, \uy\rangle_\bR$ for a choice of sign $\eta$.
In both cases the triple $(\ux_i^\epsilon, \ux_{i+1}^\eta, \uy)$ is linearly 
independent, so its determinant is a non-zero integer.

To prove (iii), fix a pair of consecutive elements $i<j$ of $I$ with $i\ge i_0$. In view of \eqref{CE:eq:Cneq0}, we have $C_{i,i+1}\neq 0$ and $C_{j-1,j}\neq 0$.
Moreover \cite[Lemma 4.2]{R2008} gives $C_{i,j}=bC_{i,i+1}$ for some non-zero 
integer $b$ with $|b|\asymp X_j/X_{i+1}$.   Thus, we have $C_{i,j}\neq 0$.  By
\eqref{CE:eq:Cneq0}, this means that $V_i\neq V_j$ and then that $C_{j,i}\neq 0$.  
\end{proof}

We conclude with two growth estimates for the sequence of norms $(X_i)_{i\ge 1}$.

\begin{lemma}
\label{CE:lemmaG}
Suppose that $\lambda > \sqrt{2}-1$.  Then, for each pair of consecutive 
elements $i<j$ of $I$, we have
\begin{equation}
\label{CE:lemmaG:eq1}
 X_{j+1} \ll X_{i+1}^\theta \quad\text{where}\quad
  \theta=\frac{1-\lambda}{\lambda}.
\end{equation}
If moreover $i\in J$, then we also have 
\begin{equation}
\label{CE:lemmaG:eq2}
 X_i \ll X_j^{\theta^2-1}.
\end{equation}
\end{lemma}

\begin{proof}
Let $i<j$ be consecutive elements of $I$.  If $i$ is large enough, we have 
$C_{j,j-1}\neq 0$ by Lemma \ref{CE:lemmaCneq0}.  Since $C_{j,j-1}\in\bZ^2$,
this implies that 
\[
 1 \le \norm{C_{j,j-1}} \ll X_jL_{j-1}L_j,
\]
where the second estimate comes from Corollary \ref{CE:corC}. As Proposition 
\ref{not:prop} gives $X_jL_{j-1}\asymp H(W_j)=H(W_{i+1})\asymp X_{i+1}L_i$,
we deduce that
\[
 1\ll X_{i+1}L_iL_j \ll X_{i+1}^{1-\lambda}X_{j+1}^{-\lambda}
\]
and \eqref{CE:lemmaG:eq1} follows. If $i\in J$, the estimate \eqref{CE:lemmaG:eq2}
follows from \cite[Corollary 5.3, Equation (11)]{R2008}.
\end{proof}

%
%

\section{The maps $\Psi_-$ and $\Psi_+$}
\label{sec:Psi}

For each choice of sign $\epsilon$ among $\{-,+\}$, we define a polynomial map
$\Psi_\epsilon\colon(\bR^4)^3\to\bR^4$ by the formula
\begin{equation}
\label{Psi:eq:Psi}
 \Psi_\epsilon(\ux,\uy,\uz)
  =C(\uy,\uz)^\epsilon\ux + E(\uy,\uz,\ux)^\epsilon\uy -C(\uy,\ux)^\epsilon\uz.
\end{equation}
We first note the following identities.

\begin{lemma}
\label{Psi:lemma}
For any choice of $\ux,\uy,\uz\in\bR^4$, we have
\begin{align*}
 \Psi_-(\ux,\uy,\uz)^- 
  &= \det(\ux^-,\uy^-,\uz^+)\uy^- - \det(\ux^-,\uy^-,\uz^-)\uy^+,\\[2pt]
 \Psi_+(\ux,\uy,\uz)^+ 
  &= \det(\ux^+,\uy^+,\uz^+)\uy^- - \det(\ux^+,\uy^+,\uz^-)\uy^+.
\end{align*}
\end{lemma}

\begin{proof}
For any choice of $\uy_1,\dots,\uy_4\in\bR^3$, we have
\[
\sum_{i=1}^4 (-1)^{i-1}\det(\uy_1,\dots,\widehat{\uy_i},\dots,\uy_4)\uy_i = 0,
\]
where $(\uy_1,\dots,\widehat{\uy_i},\dots,\uy_4)$ denotes the sequence obtained 
by removing $\uy_i$ from $(\uy_1,\dots,\uy_4)$.  The first formula follows from 
this identity applied to the points $\ux^-,\uy^-,\uy^+,\uz^-\in\bR^3$.  We obtain
the second formula by applying it to $\ux^+,\uy^-,\uy^+,\uz^+\in\bR^3$.
\end{proof}

\begin{proposition}
\label{Psi:prop1}
For any choice of $\ux,\uy,\uz\in\bR^4\setminus\{0\}$ with
\begin{equation}
\label{Psi:prop1:eq1}
 \frac{L(\ux)}{\norm{\ux}} \ge \frac{L(\uy)}{\norm{\uy}} \ge\frac{L(\uz)}{\norm{\uz}}, 
\end{equation}
and for any choice of sign $\epsilon$, we have
\begin{align}
 \norm{\Psi_\epsilon(\ux,\uy,\uz)}
  &\ll \norm{\uy}^2L(\ux)L(\uz) + \norm{\uz}L(\ux)L(\uy)^2, 
 \label{Psi:prop1:eq2}\\[2pt]
 L(\Psi_\epsilon(\ux,\uy,\uz))   &\ll \norm{\uz}L(\ux)L(\uy)^2.
 \label{Psi:prop1:eq3}
\end{align}
\end{proposition}

\begin{proof}
Fix $\ux,\uy,\uz\in\bR^4\setminus\{0\}$ with property \eqref{Psi:prop1:eq1}
and set $\psi_\epsilon=\Psi_\epsilon(\ux,\uy,\uz)$ for some sign $\epsilon$ among
$\{-,+\}$.  We first note that $\Delta\psi_\epsilon$ is a sum of four terms of the 
form
\[
 \uv = \pm\det(\uy_1^\pm,\uy_2^\pm,\uy_3^\pm)\Delta\uy_4
\]
where $(\uy_1,\uy_2,\uy_3,\uy_4)$ is a permutation of $(\ux,\uy,\uy,\uz)$ with
\[
 \frac{L(\uy_1)}{\norm{\uy_1}} 
  \ge \frac{L(\uy_2)}{\norm{\uy_2}} 
  \ge \frac{L(\uy_3)}{\norm{\uy_3}}.
\]
By the general estimate \eqref{not:eq5}, we find that
\[
 \norm{\uv} \ll \norm{\uy_3}L(\uy_1)L(\uy_2)L(\uy_4) \le \norm{\uz}L(\ux)L(\uy)^2,
\] 
and \eqref{Psi:prop1:eq3} follows.

Substituting $\uy^+=\xi\uy^-+\Delta\uy$ and $\uz^+=\xi\uz^-+\Delta\uz$ in the 
formulas of Lemma \ref{Psi:lemma}, we find
\[
 \psi_\epsilon^\epsilon 
  =\det(\ux^\epsilon,\uy^\epsilon,\Delta\uz)\uy^- 
     - \det(\ux^\epsilon,\uy^\epsilon,\uz^-)\Delta\uy.
\]
Using \eqref{not:eq5} and \eqref{Psi:prop1:eq1}, this gives 
\[
 \norm{\psi_\epsilon^\epsilon}
  \ll \norm{\uy}^2L(\ux)L(\uz) + \norm{\uz}L(\ux)L(\uy)^2,
\]
and \eqref{Psi:prop1:eq2} follows because $\norm{\psi_\epsilon} \ll
\norm{\psi_\epsilon^\epsilon} + L(\psi_\epsilon)$.
\end{proof}

\begin{corollary}
\label{Psi:cor}
For any non-zero $\uv,\uw,\ux,\uy,\uz\in\bR^4\setminus\{0\}$ with
\begin{equation}
\label{Psi:cor:eq1}
  \frac{L(\uv)}{\norm{\uv}} \ge \frac{L(\uw)}{\norm{\uw}} 
  \ge \frac{L(\ux)}{\norm{\ux}} \ge \frac{L(\uy)}{\norm{\uy}} 
 \ge\frac{L(\uz)}{\norm{\uz}}, 
\end{equation}
and for any choice of sign $\epsilon$, the integer
\[
 d_\epsilon = \det(\uv,\uw,\ux,\Psi_\epsilon(\ux,\uy,\uz))
\]
satisfies
\begin{equation}
 \label{Psi:cor:eq2}
 |d_\epsilon|
  \ll \big(\norm{\uy}^2L(\ux)L(\uz) + \norm{\ux}\,\norm{\uz}L(\uy)^2\big)
       L(\uv)L(\uw)L(\ux).
\end{equation}
\end{corollary}

\begin{proof}
In view of \eqref{Psi:cor:eq1}, the estimate \eqref{not:eq5} gives
\[
 |d_\epsilon| 
  \ll \norm{\Psi_\epsilon(\ux,\uy,\uz)}L(\uv)L(\uw)L(\ux) 
      + \norm{\ux}L(\uv)L(\uw)L(\Psi_\epsilon(\ux,\uy,\uz)).
\]
Then, \eqref{Psi:cor:eq2} follows from the estimates of the proposition.
\end{proof}

When the right hand side of \eqref{Psi:cor:eq2} is sufficiently small, the integers 
$d_-$ and $d_+$ must both be $0$.  The next proposition analyses the outcome 
of such a vanishing in a context that we will encounter later.

\begin{proposition}
\label{Psi:prop2}
Let $(\uv,\uw,\ux,\uy)$ be a basis of $\bR^4$ with $\ux^-\wedge\ux^+ \neq 0$, 
and let
\begin{equation}
\label{Psi:prop2:eq1}
 \uz=a\uy+b\ux+c\uw
\end{equation}
for some $a,b,c\in\bR$.  Suppose that 
\begin{equation}
\label{Psi:prop2:eq2}
 \det(\uv,\uw,\ux,\Psi_\epsilon(\ux,\uy,\uz)) =0 
\end{equation}
for any choice of sign $\epsilon$.  Then there exists $t\in\bR$ such that
\begin{flalign*}
 \mathrm{(i)}\ &C(\uy,\uz) = tC(\ux,\uy),  &\\
 \mathrm{(ii)}\ &C(\uz,\uy) = ct C(\ux,\uw), &\\
 \mathrm{(iii)}\ &\det(C(\uz,\ux),C(\ux,\uw)) = c^2\det(C(\uw,\ux),\, C(\ux,\uw)). &
\end{flalign*}
\end{proposition}

\begin{proof}
We will use the tri-linearity of the map $E$ as well as its properties \eqref{CE:eq:defE}
and \eqref{CE:eq:Eyxy}.  We first substitute the formula \eqref{Psi:eq:Psi} for
$\Psi_\epsilon(\ux,\uy,\uz)$ into \eqref{Psi:prop2:eq2}.  This gives
\[
 \det(\uv,\uw,\ux,\uy)E(\uy,\uz,\ux)^\epsilon - \det(\uv,\uw,\ux,\uz)C(\uy,\ux)^\epsilon
  =0,
\]
for any choice of sign $\epsilon$.  In view of \eqref{Psi:prop2:eq1}, we also have
\[
  \det(\uv,\uw,\ux,\uz) = a \det(\uv,\uw,\ux,\uy).
\]
Since $\det(\uv,\uw,\ux,\uy)\neq 0$, the first formula thus simplifies to
\[
 E(\uy,\uz,\ux)^\epsilon - aC(\uy,\ux)^\epsilon =0,
\]
which can also be rewritten as
\[
 \big(\uy^-\wedge\uz^+ - \uy^+\wedge\uz^- - a\uy^-\wedge\uy^+\big)
  \wedge\ux^\epsilon=0.
\]
As $\ux^-\wedge\ux^+\neq 0$, we therefore have
\[
 \uy^-\wedge\uz^+ - \uy^+\wedge\uz^- - a\uy^-\wedge\uy^+ = -t \ux^-\wedge\ux^+
\]
for some $t\in\bR$. This in turn implies that
\begin{equation}
\label{Psi:prop2:eq3}
 E(\uy,\uz,\uu) = aC(\uy,\uu) - tC(\ux,\uu)
\end{equation}
for any $\uu\in\bR^4$.

For the choice of $\uu=\ux$, the formula \eqref{Psi:prop2:eq3} reduces to
\begin{equation}
\label{Psi:prop2:eq4}
 E(\uy,\uz,\ux) = aC(\uy,\ux).
\end{equation}
For $\uu=\uy$, it yields formula (i) since $E(\uy,\uz,\uy)=-C(\uy,\uz)$.
For $\uu=\uz$, it gives
\begin{align*}
 C(\uz,\uy) 
   &=-aC(\uy,\uz) + tC(\ux,\uz) \\
   &=-atC(\ux,\uy) + tC(\ux,\uz)  &&\text{by (i)}\\
   &=tC(\ux,\uz-a\uy) \\
   &=tC(\ux,b\ux+c\uw) &&\text{by \eqref{Psi:prop2:eq1}}\\
   &=ctC(\ux,\uw)
\end{align*}
which is formula (ii). Upon substituting the formula \eqref{Psi:prop2:eq1} 
for $\uz$ into \eqref{Psi:prop2:eq4}, we find
\begin{align*}
 0 &=E(\uy, a\uy+b\ux+c\uw, \ux) - aC(\uy,\ux) \\
    &=aC(\uy,\ux) - bC(\ux,\uy) + cE(\uw,\uy,\ux).
\end{align*}
Using this relation, we obtain 
\begin{align*}
 C(\uz,\ux) &=(1/2)E(a\uy+b\ux+c\uw, a\uy+b\ux+c\uw, \ux) \\
   &= a^2C(\uy,\ux) - abC(\ux,\uy) + acE(\uw,\uy,\ux) - bcC(\ux,\uw) + c^2C(\uw,\ux) \\
   &= - bcC(\ux,\uw) + c^2C(\uw,\ux)
\end{align*}
which yields formula (iii).
\end{proof}

%
%

\section{First step}
\label{sec:first}

By \cite[Corollary 6.3]{R2008}, the complement $I\setminus J$ of $J$ in $I$ 
is infinite if $\lambda>\lambda_2$, where $\lambda_2\cong 0.4241$ denotes 
the positive root of the polynomial $P_2(T)=3T^4-4T^3+2T^2+2T-1$.  In fact, we 
can show that $I\setminus J$ is infinite as soon as $\lambda > (3-\sqrt{3})/3 
\cong0.4226$ but we will not go into this here as the proof is relatively elaborate.

Below, we recall the proof that $\lambda\le \lambda_3$ where $\lambda_3\cong 
0.4245$ is as in the introduction and we study in some detail the limit case where
$\lambda=\lambda_3$.  We start with a lemma which uses the notation 
$\theta=(1-\lambda)/\lambda$ from \eqref{CE:lemmaG:eq1}.

\begin{lemma}
\label{first:lemma}
For each pair of consecutive elements $k<l$ of $I$, we have 
\begin{equation}
\label{first:lemma:eq1}
 H(U_l)^{1/\lambda} \ll X_{l+1}^\theta X_{k+1}^{-1}.
\end{equation}
For a set of pairs $k<l$ for which this is optimal, namely a set of pairs $k<l$ 
satisfying $H(U_l)^{1/\lambda} \gg X_{l+1}^\theta X_{k+1}^{-1}$ for 
another implicit constant depending only on $\xi$, we have
\begin{equation}
\label{first:lemma:eq2}
 X_{k+1}\asymp X_l, 
 \quad 
 L_k \asymp X_{k+1}^{-\lambda} 
 \et 
 L_l \asymp X_{l+1}^{-\lambda}.
\end{equation}
\end{lemma}

\begin{proof}
Using the estimates of Proposition \ref{not:prop}, we find
\begin{align*}
 &H(U_l) \ll X_l^{-1}H(W_l)H(W_{l+1}) \le X_{k+1}^{-1}H(W_l)H(W_{l+1}),\\
 &H(W_l) = H(W_{k+1}) \asymp X_{k+1}L_k \ll X_{k+1}^{1-\lambda},\\
 &H(W_{l+1}) \asymp X_{l+1}L_l \ll X_{l+1}^{1-\lambda},
\end{align*}
thus $H(U_l) \ll X_{l+1}^{1-\lambda}X_{k+1}^{-\lambda}$ which is equivalent 
to \eqref{first:lemma:eq1}.  If this is optimal for a set of pairs $k<l$, then all the
above estimates are optimal for those pairs and this yields \eqref{first:lemma:eq2}.
\end{proof}

\begin{proposition}
\label{first:prop}
Suppose that $\lambda\ge \lambda_3$.  Then, we have $\lambda=\lambda_3$ 
and there are infinitely many sequences of consecutive elements $g<h<i<j$ 
of $I$ with $h\notin J$ and $i\in J$.  For each of them, we have
\begin{equation}
\label{first:prop:eq1}
\begin{array}{llll}
  X_{g+1}\asymp X_h,  \
  &X_{h+1}\asymp X_i \asymp X_h^\theta, \
  &X_{i+1}\asymp X_j \asymp X_i^{\gamma/\theta}, \
 &X_{j+1}\asymp X_j^\theta,\\
 L_g\asymp X_{g+1}^{-\lambda}, 
 &L_h\asymp X_{h+1}^{-\lambda}, 
 &L_i\asymp X_{i+1}^{-\lambda}, 
 &L_j\asymp X_{j+1}^{-\lambda},
\end{array}
\end{equation}
and
\begin{equation}
\label{first:prop:eq2}
 H(U_h)\asymp X_{h+1}^{\lambda/\gamma}.
\end{equation}
If $h$ is large enough, we also have $g\in J$ and $X_g\ll X_h^{\theta/\gamma}$.
\end{proposition}

Note that \eqref{first:prop:eq1} yields $X_{i+1}\asymp X_j \asymp X_h^\gamma$
and $X_{j+1}\asymp X_h^{\gamma\theta}$.

\begin{proof}
Since $\lambda>\lambda_2$, we know by \cite[Corollary 6.3]{R2008} that 
$I\setminus J$ is infinite.  Since $J$ is infinite as well, there are arbitrarily large
consecutive sequences of elements $g<h<i<j$ of $I$ with $h\notin J$ and 
$i\in J$.  Consider any such sequence, and set
\[
 U = W_h+W_{h+1} = W_i +W_{i+1}.
\]
Then, $U\neq U_j=W_j+W_{j+1}$ and Proposition \ref{not:prop}  gives
\begin{align}
 &H(W_j)\ll H(U) H(U_j),
 \label{first:prop:eq3}\\
 &X_jH(U_j) \ll H(W_j)H(W_{j+1}),
 \label{first:prop:eq4}\\
 &H(W_{j+1}) \asymp X_{j+1}L_j \ll X_{j+1}^{1-\lambda}.
 \label{first:prop:eq5}
\end{align}
Thus, we obtain
\begin{equation}
\label{first:prop:eq6}
 H(U) \gg \frac{H(W_j)}{H(U_j)} \gg \frac{X_j}{H(W_{j+1})}
    \gg \frac{X_j}{X_{j+1}^{1-\lambda}} 
    \ge \frac{X_{i+1}}{X_{j+1}^{1-\lambda}}.
\end{equation}
Using $1/\lambda=1+\theta$ and the estimate $X_{j+1} \ll X_{i+1}^\theta$ 
from Lemma \ref{CE:lemmaG}, this gives
\begin{equation}
\label{first:prop:eq7}
 H(U)^{1/\lambda} \gg \frac{X_{i+1}^{1+\theta}}{X_{j+1}^\theta}
  \gg X_{i+1}^{1+\theta-\theta^2}.
\end{equation}
Applying Lemma \ref{first:lemma} to $U=U_h=U_i$, we also find
\begin{align}
 &H(U)^{1/\lambda} \ll X_{i+1}^\theta X_{h+1}^{-1},
 \label{first:prop:eq8}\\
 &H(U)^{1/\lambda} \ll X_{h+1}^\theta X_{g+1}^{-1} 
    \ll X_{h+1}^{\theta-1/\theta},
 \label{first:prop:eq9}
\end{align}
where the second inequality in \eqref{first:prop:eq9} uses the estimate 
$X_{h+1}\ll X_{g+1}^\theta$ from Lemma \ref{CE:lemmaG}.  Combining
\eqref{first:prop:eq7} and \eqref{first:prop:eq8}, we obtain
\begin{equation}
\label{first:prop:eq10}
 X_{h+1} \ll X_{i+1}^{\theta^2-1},
\end{equation}
and so \eqref{first:prop:eq7} and \eqref{first:prop:eq9} yield
\begin{equation}
\label{first:prop:eq11}
 X_{i+1}^{1+\theta-\theta^2}
  \ll H(U)^{1/\lambda} \ll X_{h+1}^{\theta-1/\theta} 
  \ll X_{i+1}^{(\theta-1/\theta)(\theta^2-1)}.
\end{equation}
As $h$ can be chosen arbitrarily large, we conclude that 
\[
1+\theta-\theta^2 \le (\theta-1/\theta)(\theta^2-1)=\theta(\theta-1/\theta)^2,
\]
which can be rewritten as
\[
1 \le (\theta-1/\theta) + (\theta-1/\theta)^2.
\]
This gives $\theta-1/\theta\ge 1/\gamma$ and so 
$\lambda^2-\gamma^3\lambda+\gamma\ge 0$ which in turn implies 
that $\lambda\le \lambda_3$.  Since $\lambda\ge \lambda_3$, we conclude 
that $\lambda=\lambda_3$, thus $\theta-1/\theta=1/\gamma$ and the 
inequalities \eqref{first:prop:eq11} are optimal.  Going backwards, we 
deduce that all estimates \eqref{first:prop:eq3} to \eqref{first:prop:eq11}
are optimal.

Since \eqref{first:prop:eq8} and \eqref{first:prop:eq9} are optimal, Lemma 
\ref{first:lemma} gives   
\[
 X_{g+1}\asymp X_h,\quad  X_{h+1}\asymp X_i, \quad
 L_g\asymp X_{g+1}^{-\lambda}, \quad L_h\asymp X_{h+1}^{-\lambda}
 \et L_i\asymp X_{i+1}^{-\lambda}.
\]
 Optimality in \eqref{first:prop:eq5} 
and \eqref{first:prop:eq6} also yields $X_{i+1}\asymp X_j$ and $L_j \asymp 
X_{i+1}^{-\lambda}$.  Finally, \eqref{first:prop:eq7}, \eqref{first:prop:eq9} 
and \eqref{first:prop:eq10} being optimal, we have 
\[
 X_{j+1}\asymp X_{i+1}^\theta, \quad
 X_{h+1}\asymp X_{g+1}^\theta, \quad
 X_{h+1}\asymp X_{i+1}^{\theta^2-1}= X_{i+1}^{\theta/\gamma}
\]
and $H(U) \asymp X_{h+1}^{(\theta-1/\theta)\lambda} = X_{h+1}^{\lambda/\gamma}$.
This proves \eqref{first:prop:eq1} and \eqref{first:prop:eq2}. Finally, using Lemma \ref{first:lemma} as in \eqref{first:prop:eq9}, we find that
\[
 H(U_g) \ll X_{g+1}^{(\theta-1/\theta)\lambda} = X_{g+1}^{\lambda/\gamma}.
\]
So, if $h$ is large enough, we have $U_g\neq U =U_h$ and thus $g\in J$.  Then,
Lemma \ref{CE:lemmaG} gives $X_g\ll X_h^{\theta^2-1} = X_h^{\theta/\gamma}$.
\end{proof}

We conclude this section with three consequences of the above estimates in the limit case
where $\lambda=\lambda_3$.

\begin{corollary}
\label{first:cor1}
Suppose that $\lambda=\lambda_3$.  Then, any large enough pair of consecutive elements 
of $I$ contains at least one element of $J$.
\end{corollary}

\begin{proof}
Otherwise, since $J$ is infinite, there would exist arbitrarily large triples of consecutive 
elements $g<h<i$ of $I$ with $g\notin J$, $h\notin J$ and $i\in J$, against the last 
assertion of the proposition.
\end{proof}

More precise estimates based on similar arguments show that 
Corollary~\ref{first:cor1} holds for $\lambda\ge 0.42094$.

\begin{corollary}
\label{first:cor2}
Suppose that $\lambda=\lambda_3$, and let $h<i<j$ be consecutive elements 
of $I$ with $h\notin J$.  If $h$ is large enough, then $(\ux_h,\ux_i,\ux_j,\ux_{j+1})$
is a basis of $\bR^4$ with
\[
 1 \asymp |\det(\ux_h,\ux_i,\ux_j,\ux_{j+1})|
   \asymp X_{j+1}|\det(\Delta\ux_h,\Delta\ux_i,\Delta\ux_j)|
   \asymp X_{j+1}L_hL_iL_j
\]
\end{corollary}

\begin{proof}
For $h$ large enough, Corollary \ref{first:cor1} gives $i\in J$, and then
\[
 \bR^4 = U_i + U_j 
   = \langle\ux_h,\ux_i,\ux_j\rangle_\bR + \langle\ux_i,\ux_j,\ux_{j+1}\rangle_\bR
   = \langle\ux_h,\ux_i,\ux_j,\ux_{j+1}\rangle_\bR,
\]
thus $(\ux_h,\ux_i,\ux_j,\ux_{j+1})$ is a basis of $\bR^4$.  As this basis is made of 
integer points, its determinant $d$ is a non-zero integer.  Since
\[
 X_jL_hL_iL_{j+1} 
    \le X_jL_hL_iL_j 
   \asymp X_h^{\gamma-\lambda\theta-\lambda\gamma-\lambda\gamma\theta}
   \ll X_h^{-0.575},
\]
Corollary \ref{not:cor} gives
\[
 |d| \asymp X_{j+1}|\det(\Delta\ux_h,\Delta\ux_i,\Delta\ux_j)|
    \ll X_{j+1}L_hL_iL_j,
\]
and the conclusion follows from the computation
\[
 X_{j+1}L_hL_iL_j 
  \asymp X_h^{\gamma\theta-\lambda\theta-\lambda\gamma-\lambda\theta\gamma}
  = 1,
\]
since $\gamma\theta-\lambda\theta-\lambda\gamma-\lambda\theta\gamma
=\lambda\gamma(\theta^2-1)-\lambda\theta=0$.
\end{proof}

\begin{proposition}
\label{first:prop2}
Suppose that $\lambda=\lambda_3$, and let $g<h<i<j$ be consecutive elements 
of $I$ with $h\notin J$.  If $h$ is large enough, then 
\begin{flalign*}
&\begin{array}{rlcl}
\mathrm{(i)} &|\det(\Delta^2\ux_g,\Delta^2\ux_h)| \asymp L_gL_h
  &\text{ and } &|\det(\Delta^2\ux_i,\Delta^2\ux_j)|  \asymp L_iL_j,\\[5pt]
\mathrm{(ii)} &1\asymp\norm{C_{h,g}}\asymp L(C_{h,g})
  &\text{ and } &1\asymp\norm{C_{j,i}}\asymp L(C_{j,i}),\\[5pt]
\mathrm{(iii)} &L(C_{g,h}) \asymp X_g/X_h
  &\text{ and }  &L(C_{i,j}) \asymp X_i/X_j.
\end{array} &
\end{flalign*}
\end{proposition}

\begin{proof}
Since $g<h$ and $i<j$ are pairs of consecutive elements of $I$, 
Lemma \ref{CE:lemmaCneq0} (iii) shows that $C_{h,g}$ and $C_{j,i}$ are non-zero
points of $\bZ^2$ if $h$ is large enough.  As Corollary \ref{CE:corC} gives
\[
 \norm{C_{h,g}} \ll X_hL_gL_h\asymp X_h^{1-\lambda-\theta\lambda} = 1,
\]
we deduce that $\norm{C_{h,g}}\asymp 1$ for large enough $h$, and thus that
$L(C_{h,g}) \asymp 1$ since $\xi\notin\bQ$.  Since $L_gL_h^2$ tends to $0$ 
as $h$ goes to infinity, Lemma \ref{CE:lemmaC} yields 
\[
 1 \asymp L(C_{h,g}) \asymp X_h |\det(\Delta^2\ux_g,\Delta^2\ux_h)| 
   \ll X_hL_gL_h \asymp 1,
\]
and then 
\[
 L(C_{g,h}) \asymp X_g |\det(\Delta^2\ux_g,\Delta^2\ux_h)| 
   \asymp X_gL_gL_h,
\]
because 
$L_g/X_g$ tends to $0$ as $h$ goes to infinity.  This proves the first parts of 
(i), (ii) and (iii).  The second parts are proved in the same way. 
\end{proof}

%
%

\section{A new set of algebraic relations}
\label{sec:algC}

From now on, we assume that $\lambda=\lambda_3$ and so the estimates 
of Proposition \ref{first:prop} apply.  To alleviate the notation, we also set
\[
 C_i := C_{i,i+1} = C(\ux_i,\ux_{i+1})
\]
for each $i\in I$.  By Lemma \ref{CE:lemmaCneq0} (iii), this is a non-zero point 
of $\bZ^2$ for each large enough $i$.  In this section, we show that 
$\det(C_j,C_k)=0$ for any triple of consecutive elements $i<j<k$ of $I$
with $i\in J$ large enough, and we deduce from this that $J$ contains 
finitely many triples of consecutive elements of $I$.  By a finer analysis 
that we avoid here, one can show that this finiteness property holds whenever 
$\lambda>\lambda_2$, where $\lambda_2\cong 0.4241$ is defined at the 
beginning of Section \ref{sec:first}.

\begin{lemma}
\label{algC:lemma1}
 Let $h<i<j$ be consecutive elements of $I$ with $h\notin J$.  We have
\begin{align*}
 \norm{C_h}\ll X_h^{\theta(1-2\lambda)}, \  L(C_h)\ll X_h^{-\lambda/\gamma}, \
    \norm{C_i}\ll X_h^{\gamma(1-2\lambda)}, \
    L(C_i)\ll X_h^{-\lambda/\gamma}.
\end{align*}
Moreover, $\det(C_h,C_i)=0$ if $h$ is large enough.
\end{lemma}

\begin{proof}
The estimates of Corollary \ref{CE:corC} and Proposition \ref{first:prop} yield
\[
 \norm{C_h}\ll X_{h+1}L_h^2\asymp X_h^{\theta(1-2\lambda)}
 \et
 \norm{C_i}\ll X_{i+1}L_i^2\asymp X_h^{\gamma(1-2\lambda)}.
\]
If $h$ is large enough, Lemma \ref{CE:lemmaCneq0} (iii) gives $C_{h,i}=bC_h$ for 
some non-zero integer $b$. Then, using Corollary \ref{CE:lemmaC}, we find
\[
 L(C_h) \le L(C_{h,i}) \ll X_hL_hL_i\asymp X_h^{1-\lambda\theta-\lambda\gamma}.
\]
Similarly, if $h$ is large enough, Lemma \ref{CE:lemmaCneq0} (iii)  gives $C_{i,j}=b'C_i$ 
for some non-zero integer $b'$ and, using Corollary \ref{CE:lemmaC}, we find
\[
 L(C_i) \le L(C_{i,j}) \ll X_iL_iL_j
    \asymp X_h^{\theta-\lambda\gamma-\lambda\theta\gamma}.
\]
This proves the first row of estimates since
\[
 1-\lambda\theta-\lambda\gamma=\lambda-\lambda\gamma=-\lambda/\gamma
 \et
 \theta-\lambda\gamma-\lambda\theta\gamma
   =\theta-\gamma=-\lambda/\gamma.
\]
Finally, using these estimates, Lemma \ref{not:lemma3} gives
\[
  |\det(C_h,C_i)|\ll \norm{C_h}L(C_i)+\norm{C_i}L(C_h)
  \ll X_h^a,
\]
where $a=\gamma(1-2\lambda)-\lambda/\gamma<-0.018$.  As $\det(C_h,C_i)$ is an
integer, it must be $0$ if $h$ is large enough.
\end{proof}

\begin{lemma}
\label{algC:lemma2}
 Let $i<j<k$ be consecutive elements of $I$ with $i\in J$.  If $i$ is large enough, 
we have $\det(C_j,C_k)=0$. 
\end{lemma}

\begin{proof}
If $j\notin J$, this follows from Lemma \ref{algC:lemma1}.  So, we may assume 
that $j\in J$.  Then, we have $\{i,j\}\subset J$ and \cite[Lemma 6.1]{R2008} 
gives
\begin{equation}
\label{algC:lemma2:eq1}
 L(C_j) \ll X_{k+1}^\alpha \ \text{ where } \ 
 \alpha=\frac{-\lambda^4+\lambda^3+\lambda^2-3\lambda+1}%
                    {\lambda(\lambda^2-\lambda+1)}
          \cong -0.1536.
\end{equation}
If $k\in J$, we also have $\{j,k\}\subset J$ and the same result gives
$L(C_k) \ll X_{l+1}^\alpha$ where $l$ is the successor of $k$ in $I$, 
and so, a fortiori,
\begin{equation}
\label{algC:lemma2:eq2}
 L(C_k) \ll X_{j+1}^\alpha.
\end{equation}
If $k\notin J$, the last estimate still holds as Lemma \ref{algC:lemma1}
gives $L(C_k) \ll X_k^{-\lambda/\gamma}\le X_{j+1}^{-\lambda/\gamma}$
where $-\lambda/\gamma\cong -0.2623 <\alpha$.  Using \eqref{algC:lemma2:eq1}
and \eqref{algC:lemma2:eq2} together with the estimates for $\norm{C_j}$
and $\norm{C_k}$ coming from Corollary~\ref{CE:corC}, 
Lemma~\ref{not:lemma3} gives
\[
 |\det(C_j,C_k)| \ll \norm{C_j}L(C_k) + \norm{C_k}L(C_j)
   \ll X_{j+1}^{1-2\lambda+\alpha} + X_{k+1}^{1-2\lambda+\alpha}
   \ll X_{j+1}^{-0.0026}.
\]
Thus $\det(C_j,C_k)=0$ if $i$ is large enough.
\end{proof}

\begin{proposition}
\label{algC:prop}
 The set $J$ contains finitely many triples of consecutive elements of $I$. 
\end{proposition}

\begin{proof}
Suppose on the contrary that $J$ contains infinitely many such triples.
Then there are infinitely many maximal sequences of consecutive elements
$i<j<\cdots<r$ of $I$ contained in $J$, whose cardinality is at least $3$.
If $i$ is large enough, such a sequence extends to a sequence
\[
 h<i<j<\cdots<r<h'<i'
\]
of consecutive elements of $I$ with $h\notin J$ and $h'\notin J$, and by 
Lemma~\ref{algC:lemma2}, the integer points $C_j,\dots,C_r,C_{h'},C_{i'}$ 
are all integral multiples of a single primitive point $C$ of $\bZ^2$.  Using
Corollary \ref{CE:corC} and Lemma \ref{algC:lemma1}, we find that
\[
 \norm{C}\le \norm{C_j} \ll X_{j+1}^{1-2\lambda} 
 \et
L(C) \le L(C_{i'}) \ll X_{h'}^{-\lambda/\gamma}.
\]
As $r\in J$, Lemma \ref{CE:lemmaG} gives $X_r \ll X_{h'}^{\theta^2-1}
=X_{h'}^{\theta/\gamma}$.  As $r>j$, we also have  $X_r\ge X_{j+1} \asymp 
X_j^\theta$ using the estimates of Proposition \ref{first:prop}. Thus, we obtain
\[
 L(C) \ll X_r^{-\lambda/\theta} \ll X_j^{-\lambda}.
\]
We form the point
\[
 \uy = C^-\ux_j^+ - C^+\ux_j^- \in \bZ^3.
\]
If $h$ is large enough, then $V_j=\langle \ux_j^-,\ux_j^+\rangle_\bR$ has 
dimension $2$ by Lemma \ref{CE:lemmaCneq0} (i), and so $\uy$ is non-zero.
Using Lemma \ref{not:lemma3}, we find
\begin{align*}
 &L(\uy) \ll \norm{C}L_j \ll X_{j+1}^{1-3\lambda},\\
 &\norm{\uy} \ll X_jL(C)+\norm{C}L_j \ll X_j^{1-\lambda}
\end{align*}
since $1-3\lambda<0$.  So, for any choice of signs $\epsilon$ and $\eta$,
we obtain, using the general estimate \eqref{not:eq5},
\begin{align*}
 |\det(\ux_{h-1}^\epsilon,\ux_{h}^\eta,\uy)|
  &\ll X_hL_{h-1}L(\uy) + \norm{\uy}L_{h-1}L_h\\
  &\ll X_h^{1-\lambda+\gamma\theta(1-3\lambda)} 
       + X_h^{\gamma(1-\lambda)-\lambda-\theta\lambda}
    \ll X_h^{-0.024}.  
\end{align*}
By Lemma \ref{CE:lemmaCneq0} (ii), this is impossible if $h$ is large enough.
\end{proof}

%
%

\section{Another set of algebraic relations}
\label{sec:another}

As in the preceding section, we assume that $\lambda=\lambda_3\cong 0.4245$.  
We start with the following observation.

\begin{lemma}
\label{another:lemma1}
Let $g<h<i<j$ be consecutive elements of $I$ with $h\notin J$.  Then we have
\[
 p\ux_j = q\ux_i + r\ux_h + s\ux_g
\]
for integers $p$, $q$, $r$, $s$ with $1\le |p|\ll 1$ and $1\le |s|\ll 1$.  Moreover, if 
$h$ is large enough, then $(\ux_{g-1},\ux_g,\ux_h,\ux_i)$ is a basis of $\bR^4$.
\end{lemma}

\begin{proof}
Set $U=U_h=U_i$.  Then $(\ux_g,\ux_h,\ux_i)$ and $(\ux_h,\ux_i,\ux_j)$ are bases
of $U$ as a vector space over $\bR$, while $(\ux_h,\ux_i)$ is a basis of $W_{h+1}=W_i$ 
over $\bR$.  

By Proposition \ref{not:prop} (i), the pair $(\ux_h,\ux_{h+1})$ is a basis 
of $W_{h+1}\cap\bZ^4$ over $\bZ$.  Thus, it can be extended to a basis 
$(\ux_h,\ux_{h+1},\uy)$ of $U\cap\bZ^4$ over $\bZ$.  By the above, we can write
\begin{align*}
 \ux_i &= a\ux_h+b\ux_{h+1},\\
 \ux_g &= a'\ux_h+b'\ux_{h+1}+c'\uy,\\
 \ux_j &= a''\ux_h+b''\ux_{h+1}+c''\uy,
\end{align*}
for a unique choice of integers $a,a',a'',b,b',b'',c',c''$ with $b\neq 0$, 
$c'\neq 0$ and $c''\neq 0$.  For these integers, we find that
\begin{equation}
\label{another:lemma1:eq}
 bc'\ux_j-bc''\ux_g\in\langle\ux_h,\ux_i\rangle_\bZ.
\end{equation}
We claim that $|bc'|\ll  1$ and $|bc''|\ll 1$.  To prove this, we note that 
$\ux_h\wedge\ux_i=b\ux_h\wedge\ux_{h+1}$, thus
\[
 \norm{\ux_g\wedge\ux_h\wedge\ux_i} 
  = \norm{b\ux_g\wedge\ux_h\wedge\ux_{h+1}}
  = \norm{bc'\uy\wedge\ux_h\wedge\ux_{h+1}}
  = |bc'| H(U).
\]
Similarly, we find that
\[
 \norm{\ux_j\wedge\ux_h\wedge\ux_i} = |bc''| H(U).
\]
The claim then follows from the following computations based on the general 
estimate \eqref{not:eq6} and the estimates of Proposition \ref{first:prop}, namely
\begin{align*}
 &\norm{\ux_g\wedge\ux_h\wedge\ux_i}
    \ll X_i L_g L_h \asymp X_h^{\theta-\lambda-\lambda\theta}
    = X_h^{\lambda\theta/\gamma} \asymp H(U),\\
 &\norm{\ux_j\wedge\ux_h\wedge\ux_i}
    \ll X_j L_h L_i \asymp X_h^{\gamma-\lambda\theta-\lambda\gamma}
    = X_h^{\lambda\theta/\gamma} \asymp H(U),
\end{align*}
because $\theta-\lambda-\lambda\theta=\lambda(\theta^2-1)
=\lambda\theta/\gamma$ and $\gamma-\lambda\theta-\lambda\gamma
=(1-\lambda)(\gamma-1)=\lambda\theta/\gamma$.   This claim together 
with \eqref{another:lemma1:eq} proves the first assertion of the lemma.

Finally, if $h$ is large enough, Proposition \ref{first:prop} gives $g\in J$, thus 
$U_g+U_h=\bR^4$.  Since  $(\ux_{g-1},\ux_g,\ux_h)$ is a basis of $U_g$ 
while $(\ux_g,\ux_h,\ux_i)$ is a basis of $U_h$, it follows that 
$(\ux_{g-1},\ux_g,\ux_h,\ux_i)$ is then a basis of $\bR^4$.
\end{proof}

The next result plays a crucial role and holds whenever $\lambda\ge 0.42094$.
Here, we only prove it under our current hypothesis that $\lambda=\lambda_3$.

\begin{proposition}
\label{another:prop} 
Let $g<h<i<j$ be consecutive elements of $I$ with $h\notin J$, and let
$\epsilon$ be a sign among $\{-,+\}$.  If $h$ is large enough, we have
\begin{equation}
\label{another:prop:eq}
 \det(\ux_{g-1},\ux_g,\ux_h,\Psi_\epsilon(\ux_h,\ux_i,\ux_j))=0.
\end{equation}
\end{proposition}

\begin{proof}
The conditions \eqref{Psi:cor:eq1} of Corollary \ref{Psi:cor} are fulfilled for the sequence
$(\uv,\uw,\ux,\uy,\uz)=(\ux_{g-1},\ux_g,\ux_h,\ux_i,\ux_j)$.  So, upon denoting by 
$d_\epsilon$ the determinant in the left hand side of \eqref{another:prop:eq}, we obtain
\[
 |d_\epsilon| \ll (X_i^2L_hL_j+X_hX_jL_i^2)L_{g-1}L_gL_h.
\]
Using the estimates \eqref{first:prop:eq1} of Proposition \ref{first:prop}, we find
\[ 
 X_i^2L_hL_j  \asymp X_h^{2\theta-\lambda\theta-\lambda\gamma\theta}
   \le X_h^{1.2047}
 \et
 X_hX_jL_i^2 \asymp X_h^{1+\gamma-2\lambda\gamma} \le X_h^{1.2444}.
\] 
Since Lemma \ref{CE:lemmaG} gives $X_{g+1}\ll X_g^\theta$, we also find that
$L_{g-1}\ll X_g^{-\lambda}\ll X_{g+1}^{-\lambda/\theta}$, thus
\[
 L_{g-1}L_gL_h 
  \ll X_h^{-\lambda/\theta-\lambda-\lambda\theta} \le X_h^{-1.3131},
\]
and so $|d_\epsilon|\ll X_h^{-0.687}$. As $d_\epsilon$ is an integer, we conclude 
that $d_\epsilon=0$ if $h$ is large enough.
\end{proof}

\begin{corollary}
\label{another:cor} 
Let $g<h<i<j$ be consecutive elements of $I$ with $h\notin J$.  If $h$ is 
large enough, there are non-zero rational numbers $c$ and $t$ whose numerators
and denominators are bounded only in terms of $\xi$, such that
\begin{flalign*}
&\begin{array}{rl}
 \mathrm{(i)}   &C_{i,j}=tC_{h,i},\\[5pt]
 \mathrm{(ii)}  &C_{j,i}=ctC_{h,g},\\[5pt]
 \mathrm{(iii)} &\det(C_{j,h},C_{h,g})=c^2\det(C_{g,h},C_{h,g}).
\end{array}&
\end{flalign*}
\end{corollary}

\begin{proof}
Lemma \ref{another:lemma1} and Proposition \ref{another:prop} show that the
hypotheses of Proposition \ref{Psi:prop2} are fulfilled with 
$(\uv,\uw,\ux,\uy,\uz)=(\ux_{g-1},\ux_g,\ux_h,\ux_i,\ux_j)$ and $c=s/p$ 
for bounded non-zero integers $p$ and $s$,
if $h$ is large enough. Then (i), (ii) and (iii) hold for some $t\in\bR$.  If $h$ is
large enough, Proposition \ref{first:prop2} (ii) also gives $\norm{C_{j,i}}
\asymp\norm{C_{h,g}}\asymp 1$.  Then (ii) implies that $ct$ is a non-zero rational
number with bounded numerator and denominator.  Since $c$ has the same property, 
this applies to $t$ as well.
\end{proof}

The third identity of the corollary has the following consequence.

\begin{lemma}
\label{another:lemma2}
Let $g<h<i<j$ be consecutive elements of $I$ with $h\notin J$.  If $h$ is 
sufficiently large, we have
\[
 \norm{C_g} \asymp |\det(C_{j,h},C_{h,g})|    
            \ll \norm{C_{j,h}} \ll X_h^{\lambda^2/\gamma}.
\]
\end{lemma}

As $\lambda^2/\gamma \cong 0.111$, this is a significant improvement on the
generic upper bound $\norm{C_g} \ll X_{g+1}^{1-2\lambda}\asymp X_h^{1-2\lambda}$
coming from Corollary \ref{CE:corC}, where $1-2\lambda\cong 0.151$.

\begin{proof}[Proof of Lemma \ref{another:lemma2}]
If $h$ is large enough, Proposition \ref{first:prop2} (ii) gives $1 \asymp \norm{C_{h,g}} 
\asymp L(C_{h,g})$, and Lemma \ref{CE:lemmaCneq0} (iii) gives $C_{g,h}=bC_g$ 
for some non-zero integer $b$ with $|b|\asymp X_h/X_{g+1} \asymp 1$.  Thus, if $h$
is sufficiently large, Corollary \ref{another:cor} (iii) yields
\[
 |\det(C_g,C_{h,g})| \asymp |\det(C_{j,h},C_{h,g})| \ll \norm{C_{j,h}}.
\]
Using Corollary \ref{CE:corC}, we also find
\[
 \norm{C_{j,h}} \ll X_jL_jL_h 
  \asymp X_h^{\gamma-\lambda\gamma\theta-\lambda\theta}
  = X_h^{\lambda^2/\gamma}
\]
since $\gamma-\lambda\gamma\theta-\lambda\theta=-1+\gamma^2\lambda
=\lambda^2/\gamma$.  On the other hand, we note that
\[
 L(C_g) = |b|^{-1}L(C_{g,h}) 
   \asymp X_g/X_h \ll X_h^{\theta/\gamma-1} \le X_h^{-0.162}
\]
using Proposition \ref{first:prop2} (iii) and the estimate 
$X_g\ll X_h^{\theta/\gamma}$ of Proposition \ref{first:prop}.  
In particular, this means that $\norm{C_g\wedge(1,\xi)} 
\asymp L(C_g)$ tends to $0$ as $h\to\infty$.  
As $\norm{C_{h,g}\wedge(1,\xi)} \asymp L(C_{h,g}) \asymp 1$,
we conclude that the angle between $C_g$ and $C_{h,g}$ is bounded away from 
$0$ as $h\to\infty$ and so
\[
 |\det(C_g,C_{h,g})| \asymp \norm{C_g}\norm{C_{h,g}} \asymp \norm{C_g}.
\qedhere
\]
\end{proof}

\begin{proposition}
\label{another:prop2}
Any sufficiently large pair of consecutive elements of $I$ contains exactly one element of $J$.
\end{proposition}

\begin{proof}
By Corollary \ref{first:cor1}, any sufficiently large pair of consecutive elements 
of $I$ contains at least one element of $J$.  So, it remains to show that $J$ 
contains finitely many pairs of consecutive elements of $I$.

Suppose on the contrary that $J$ contains infinitely many such pairs.  Then it follows 
from Proposition \ref{algC:prop} and Corollary \ref{first:cor1} that there exist 
arbitrarily large sequences of consecutive elements $g < h < i < j < k  < l$ of $I$ 
with 
\[
 g\in J, \quad h\notin J, \quad i\in J, \quad j\in J, \quad k\notin J, \quad l\in J.
\]
Since $k\notin J$, Lemma \ref{another:lemma2} gives 
\begin{equation}
\label{another:prop2:eq1}
 \norm{C_j} \ll X_k^{\lambda^2/\gamma}.
\end{equation}
On the other hand, if $h$ is large enough, Lemma \ref{CE:lemmaCneq0} (iii) gives
$C_{j,k}=bC_j$ for some non-zero $b\in\bZ$ with $|b|\asymp X_k/X_{j+1}\asymp 1$.
In view of this, Proposition \ref{first:prop2} (iii) gives
\begin{equation}
\label{another:prop2:eq2}
 L(C_j) \asymp L(C_{j,k}) 
     \asymp X_j/X_k \asymp X_k^{1/\theta-1} = X_k^{-\lambda/\gamma},
\end{equation}
using the fact that $X_k\asymp X_{j+1}\asymp X_j^\theta$ since $h\notin J$ and 
$k\notin J$.  
Combining \eqref{another:prop2:eq1} and \eqref{another:prop2:eq2}, we obtain 
$L(C_j) \ll \norm{C_j}^{-1/\lambda}$.  By \cite[Lemma 2.2]{R2008}, this
implies that $L(C_j) \asymp \norm{C_j}^{-1/\lambda}$, but we will not need that.
We will get the desired contradiction by considering the sequence $e<f<g<h$ of
four consecutive elements of $I$ ending with $h$, and by forming the point
\[
 \uy = C_j^-\ux_e^+ - C_j^+\ux_e^- \in \bZ^3.
\]
If $h$ is large enough, Lemma \ref{CE:lemmaCneq0} (i) shows that the points 
$\ux_e^-$ and $\ux_e^+$ are linearly independent and thus $\uy$ is non-zero.
By Lemma \ref{not:lemma3}, we have 
\begin{equation}
\label{another:prop2:eq3}
 \norm{\uy} \ll \norm{C_j}L_e + X_e L(C_j).
\end{equation}
If $f\notin J$, Proposition \ref{first:prop} gives, for $h$ large enough,
\begin{equation}
\label{another:prop2:eq4}
 X_e\ll X_f^{\theta/\gamma}, \quad 
 L_e\asymp X_f^{-\lambda} \et
 X_f^\gamma\asymp X_h.
\end{equation}
If $f\in J$ and $h$ is large enough, Proposition \ref{algC:prop} tells us 
that $e\notin J$ because $f,g\in J$.  Then, Proposition \ref{first:prop} 
shows that the estimates \eqref{another:prop2:eq4} still hold.  In fact, it
even gives the stronger estimate $X_e \asymp X_f^{1/\theta}$ with exponent
$1/\theta<\theta/\gamma$.  Combining 
\eqref{another:prop2:eq1}--\eqref{another:prop2:eq4} and using the estimate
$X_k\asymp X_{j+1}\asymp X_h^{\gamma\theta}$ coming from 
Proposition \ref{first:prop}, we find that
\[
 \norm{\uy} \ll X_k^{\lambda^2/\gamma}X_h^{-\lambda/\gamma}
 + X_h^{\theta/\gamma^2} X_k^{-\lambda/\gamma}
  \ll X_h^{\lambda^2\theta-\lambda/\gamma}
 + X_h^{\theta/\gamma^2-\lambda\theta}
  \ll X_h^{-0.018}.
\]
For $h$ large enough, this is impossible as $\uy\neq 0$.
\end{proof}

%
%

\section{Final contradiction}
\label{sec:final}

In this section, we assume that our fixed real number $\xi$
of Section \ref{sec:not} satisfies the hypotheses of 
Theorem \ref{intro:thm1} for $\lambda=\lambda_3\cong 0.4245$
and we prove Theorem \ref{intro:thm2} by reaching a contradiction.

More precisely, we will show that if $f<g<h<i<j<k<l$ are consecutive 
elements of $I$ with $h\notin J$ large enough, then the points
$C_{f,h}$ and $C_{k,l}$ are linearly dependent with
\begin{equation}
\label{final:eq1}
    \norm{C_{f,h}} < \norm{C_{k,l}}
    \et
    L(C_{f,h}) > L(C_{k,l})
\end{equation}
which is impossible. To show this, we will need sharp estimates 
on the above quantities.

Proposition \ref{another:prop2} greatly simplifies the problem 
by showing that large consecutive elements of $I$ alternate 
between $J$ and $I\setminus J$.  By Proposition 
\ref{first:prop}, this provides sharp estimates on the minimal 
points.  Explicitly, if $h<i<j$ are large consecutive elements 
of $I$ with $h\notin J$, then Proposition \ref{another:prop2} shows
that $i\in J$ and that $j\notin J$, and Proposition \ref{first:prop}
gives
\begin{equation}
    \label{final:eq2}
    X_{h+1}\asymp X_i \asymp X_h^\theta, \quad
    X_{i+1}\asymp X_j \asymp X_i^{\gamma/\theta}, \quad 
    L_h \asymp X_{h+1}^{-\lambda}, \quad
    L_i \asymp X_{i+1}^{-\lambda}.
\end{equation}
In particular, this gives $X_i^{\gamma/\theta}\ll X_{i+1}\ll X_i^\theta$
for each $i\in I$. 
Corollary \ref{another:cor} (ii) also has the following consequence.

\begin{lemma}
\label{final:lemma0}
There is a primitive point $(a,b)\in \bZ^2$ such that 
\[
  C_{h,g} \in \langle (a,b) \rangle_\bZ
\]
for each large enough pair of consecutive elements $g<h$ of $I$
with $h\notin J$.
\end{lemma}

\begin{proof}
For each large enough pair of consecutive elements $g<h$ of $I$
with $h\notin J$, the next pair of consecutive elements $i<j$ of $I$ 
has $j\notin J$, and Corollary \ref{another:cor} (ii) shows that
$C_{h,g}$ and $C_{j,i}$ are linearly dependent. As the latter are 
non-zero points of $\bZ^2$, they are integer multiples of the same 
primitive point $(a,b)$ of $\bZ^2$.  The result follows by induction
on $h$.
\end{proof}

For each integer $i\ge 1$, we define
\[
\wdeltax_i = \frac{\Delta\ux_i}{\norm{\Delta\ux_i}}
\et
\wddeltax_i 
= \frac{\Delta^2\ux_i}{\norm{\Delta\ux_i}} = \Delta(\wdeltax_i).
\]
Since $\norm{\Delta\ux_i}\asymp L_i$, Corollary \ref{first:cor2} and 
Proposition \ref{first:prop2} (i) have the following immediate consequences.

\begin{lemma}
\label{final:lemma1}
For any large enough consecutive elements $g<h<i<j$ of $I$ with 
$h\notin I$, we have
\[
 |\det( \wdeltax_h, \wdeltax_i, \wdeltax_j ) | \asymp 1
 \et
 |\det( \wddeltax_g, \wddeltax_h ) | \asymp 1.
\]
\end{lemma}

The next lemma asks for precise estimates for the quantities
$|\det(\wdeltax_i^-,\wdeltax_i^+)|$ as $i$ goes to infinity in $I$.

\begin{lemma}
\label{final:lemma2}
For any large enough integers $i<j$ with $i\in I$, we have
\[
 1 \le \norm{C_{i,j}} 
    \asymp \frac{X_j}{X_{i+1}} \norm{C_i}
    \asymp X_jL_i^2 |\det( \wdeltax_i^-, \wdeltax_i^+ ) |.
\]
\end{lemma}

\begin{proof}
For integers $1\le i < j$, Lemma \ref{CE:lemmaC} gives
\[
\norm{C_{i,j}} = c\, |x_{j,0}|\,|\det(\Delta\ux_i^-,\Delta^2\ux_i)|
    + \cO(X_iL_iL_j)
\]
where $c=\max\{1,|\xi|\}$ and where $x_{j,0}$ is the first coordinate 
of $\ux_j$.  If $i\in I$, we also have $X_{i+1}\gg X_i^{\gamma/\theta}$
by the remark below \eqref{final:eq2}, thus 
\[ 
  X_iL_iL_j \le X_iL_i^2 \ll X_i^{1-2\lambda\gamma/\theta} \ll X_i^{-0.0133}.
\]
As $\Delta^2\ux_i = \Delta\ux_i^+ - \xi\Delta\ux_i^-$, we deduce that
\[
\norm{C_{i,j}} = c\,|x_{j,0}|\,|\det(\Delta\ux_i^-,\Delta\ux_i^+)|
    + \cO(X_i^{-0.0133}).    
\]
Moreover, if $i$ is large enough, Lemma \ref{CE:lemmaCneq0} shows that
$C_i=C_{i,i+1}$ is a non-zero point of $\bZ^2$.  Then the above estimate with 
$j=i+1$ yields
\[
  1\le \norm{C_i} \asymp X_{i+1} |\det(\Delta\ux_i^-,\Delta\ux_i^+)|,
\]
and the conclusion follows.
\end{proof}

We now exploit the various estimates of Corollary \ref{another:cor} and 
their consequences developed in Lemmas \ref{another:lemma2} and 
\ref{final:lemma0}.

\begin{proposition}
\label{final:prop1}
Let $(a,b)$ be as in Lemma \ref{final:lemma0}.  For any large enough 
consecutive elements $g<h<i<j$ of $I$ with $h\notin J$, we have
\begin{itemize}
\item[\textrm{(i)}] 
 $|\det(\wdeltax_h^-,\wdeltax_h^+)|
   \asymp X_h^\sigma |\det(\wdeltax_i^-,\wdeltax_i^+)|$,
\smallskip
\item[\textrm{(ii)}] 
 $|\det(\wddeltax_h, a\wdeltax_g^+ - b\wdeltax_g^-)|
\ll X_h^{\gamma\lambda-1} \ll X_h^{-0.3131}$,
\smallskip
\item[\textrm{(iii)}] 
 $|\det(\wdeltax_g^-, \wdeltax_g^+)|
   \asymp X_h^{-\sigma} |\det(\wddeltax_j, a\wdeltax_h^+ - b\wdeltax_h^-)|$,
\smallskip
\end{itemize}
where $\sigma=2-(3+\gamma)\lambda\cong 0.0396$.
\end{proposition}

\begin{proof}
By Corollary \ref{another:cor} (i), we have $\norm{C_{h,i}}\asymp\norm{C_{i,j}}$,
and thus
\[
 X_iL_h^2|\det(\wdeltax_h^-,\wdeltax_h^+)|
   \asymp X_jL_i^2 |\det(\wdeltax_i^-,\wdeltax_i^+)|
\]
by the previous lemma.
Since $X_iL_h^2\asymp X_i^{1-2\lambda}\asymp X_h^{\theta(1-2\lambda)}$
and $X_jL_i^2\asymp X_j^{1-2\lambda}\asymp X_h^{\gamma(1-2\lambda)}$, 
this yields the estimate of part (i) with $\sigma= (\gamma-\theta)(1-2\lambda)
=(\lambda/\gamma)(1-2\lambda)=2-(3+\gamma)\lambda$.

Lemma \ref{final:lemma0} implies that $aC_{h,g}^+ - bC_{h,g}^-=0$.
Using the formulas of Lemma \ref{CE:lemmaC}, this gives 
\[
 X_h |\det(\Delta^2\ux_h, a\Delta\ux_g^+-b\Delta\ux_g^-)| \ll X_gL_h^2,
\]
and part (ii) follows since $X_gL_h^2/(X_hL_gL_h) = X_gL_h/(X_hL_g) \asymp 
X_h^{\theta/\gamma-\lambda\theta-1+\lambda} = X_h^{\gamma\lambda-1}$.

Finally, Lemma \ref{another:lemma2} gives 
$\norm{C_g} \asymp |\det(C_{j,h},C_{h,g})|$.  As $C_{h,g}$ 
is a non-zero multiple of $(a,b)$ by Lemma \ref{final:lemma0}, and 
as it has bounded norm by Proposition \ref{first:prop2} (ii), it is a bounded
non-zero multiple of $(a,b)$.  We deduce that
\begin{align*}
 1 \le \norm{C_g}    
   \asymp |aC_{j,h}^+ - bC_{j,h}^-|
   &= |\det(\ux_j^-, \Delta\ux_j, a\ux_h^+-b\ux_h^-)| \\
   &\asymp X_j |\det(\Delta^2\ux_j, a\Delta\ux_h^+-b\Delta\ux_h^-)|,
\end{align*}
using Lemma \ref{not:lemma1} to expand the determinant, and noting 
that $X_hL_j^2\to 0$ as $h\to\infty$.  Since Lemma \ref{final:lemma2}
gives $\norm{C_g}\asymp X_{g+1} |\det(\Delta\ux_g^-, \Delta\ux_g^+)|$,
we obtain the estimate of part (iii) by observing that
$X_jL_hL_j/(X_{g+1}L_g^2) \asymp X_h^{\gamma-\theta\lambda-\gamma\theta\lambda-1+2\lambda} = X_h^{-\sigma}$.
\end{proof}

In a first step, we deduce upper bound estimates for the quantities
$|\det(\wdeltax_i^-,\wdeltax_i^+)|$ with $i\in I$.  We will show later
that they are best possible up to multiplicative constants.

\begin{corollary}
\label{final:cor1}
Let $\sigma$ be as in Proposition \ref{final:prop1}.  For any pair of consecutive
elements $g<h$ of $I$ with $h\notin J$, we have
\[
 \mathrm{(i)} \quad  |\det(\wdeltax_g^-,\wdeltax_g^+)| \ll X_h^{-\sigma}
 \qquad\text{and}\qquad
 \mathrm{(ii)} \quad |\det(\wdeltax_h^-,\wdeltax_h^+)| \ll X_h^{-\sigma/\gamma}.
\]
\end{corollary}

\begin{proof}
We may assume that $g<h$ are large enough so that Proposition \ref{final:prop1} 
applies to the sequence of four consecutive elements $g<h<i<j$ of $I$ starting
with $g$.  Then part (i) follows immediately from Proposition \ref{final:prop1} (iii).  
For part (ii), we may assume that $h$ is large enough 
so that $j\notin J$ and thus the estimate of part (i) holds with the pair $g<h$ 
replaced by $i<j$.  Then Proposition \ref{final:prop1} (i) gives
\[
 |\det(\wdeltax_h^-,\wdeltax_h^+)| 
   \ll X_h^{\sigma}X_j^{-\sigma} 
  \asymp X_h^{\sigma-\gamma\sigma} = X_h^{-\sigma/\gamma}.
\qedhere
\]
\end{proof}

\begin{corollary}
\label{final:cor2}
Let $\sigma$ be as in Proposition \ref{final:prop1}.  For any pair of consecutive
elements $g<h$ of $I$ with $h\notin J$, there are points $(s_g,t_g)$
and $(s_h,t_h)$ of norm $1$ in $\bR^2$ such that
\[
 \wdeltax_g = \pm (s_g^2, s_gt_g, t_g^2) + \cO(X_h^{-\sigma})
 \et
 \wdeltax_h = \pm (s_h^2, s_ht_h, t_h^2) + \cO(X_h^{-\sigma/\gamma}).
\]
\end{corollary}

As $\wdeltax_g$ and $\wdeltax_h$ are points of norm $1$ in $\bR^3$, this
is a direct consequence of Corollary \ref{final:cor1} and of the following simple 
observation.

\begin{lemma}
\label{final:lemma:det}
Let $\uy\in\bR^3$ with $\norm{\uy}=1$, and let $\delta=|\det(\uy^-,\uy^+)|$.
There exists a point $(r,s)\in\bR^2$ with $\norm{(r,s)}=1$ such that
\[
   \norm{\uy\pm(r^2,rs,s^2)} \le 2\delta.
\]
\end{lemma}

\begin{proof}
We may assume that $\delta<1$, otherwise any point $(r,s)$ of norm $1$
has the required property.  Writing $\uy=(a,b,c)$, we have 
$\delta=|ac-b^2|$.  By permuting $a$ and $c$, and by multiplying $\uy$
by $-1$ if necessary, we may assume that $a=|a|\ge |c|$.  We set
$(r,s)=(1,b)$. Then we have $\norm{(r,s)}=1$ since $|b|\le \norm{\uy}\le 1$ 
and we find
\begin{equation}
\label{final:lemma:det:eq}
 \norm{\uy - (1,b,b^2)} = \max\{1-a, |b^2-c|\}.
\end{equation}
If $a<1$, we have $|c|<1$, thus $|b|=1$ since $\norm{\uy}=1$, and then
$\delta=1-ac$. As $\delta<1$, this implies that $c>0$, thus the right 
hand side of \eqref{final:lemma:det:eq} becomes $\max\{1-a, 1-c\}\le 
\delta$. If $a=1$, it reduces to $|b^2-c|=\delta$.  In both cases, 
we are done.
\end{proof}

From now on, we fix a pair of points $(s_g,t_g)$ and $(s_h,t_h)$ 
as in Corollary \ref{final:cor2} for each pair of consecutive elements $g<h$ of $I$
with $g\in J$ and $h\notin J$.  This yields a unique point $(s_i,t_i)$ for each large 
enough $i\in I$.    

\begin{proposition}
\label{final:prop2}
For each large enough sequence of consecutive elements $g<h<i<j$ of $I$ with
$h\notin J$, we have 
\begin{flalign*}
&\begin{array}{rl}
 \mathrm{(i)}   &1\asymp |t_g-\xi s_g| \asymp  |t_h-\xi s_h| \asymp |s_gt_h-s_ht_g|,\\[5pt]
 \mathrm{(ii)}  &1 \asymp |s_ht_i-s_it_h| \asymp |s_ht_j-s_jt_h| \asymp |s_it_j-s_jt_i|.
\end{array}&
\end{flalign*}
\end{proposition}

\begin{proof}
Using the formulas of Corollary~\ref{final:cor2}, the estimates of 
Lemma~\ref{final:lemma1} become
\begin{align*}
 &1\asymp  |\det( \wddeltax_g, \wddeltax_h ) |
  = \Big| (t_g-\xi s_g) (t_h-\xi s_h) 
              \det\begin{pmatrix} s_g&t_g\\ s_h&t_h\end{pmatrix}\Big|
     + \cO(X_h^{-\sigma/\gamma}), \\
 &1\asymp |\det( \wdeltax_h, \wdeltax_i, \wdeltax_j ) | 
  = \Bigg| 
     \det\begin{pmatrix} 
          s_h^2 &s_h t_h &t_h^2\\ 
          s_i^2 &s_i t_i &t_i^2\\
          s_j^2 &s_j t_j &t_j^2\\\end{pmatrix}\Bigg|
     + \cO(X_h^{-\sigma/\gamma})\\
  &\phantom{1\asymp |\det( \wdeltax_h, \wdeltax_i, \wdeltax_j ) |}
   =  |(s_ht_i-s_it_h)(s_ht_j-s_jt_h)(s_it_j-s_jt_i)| + \cO(X_h^{-\sigma/\gamma}).
\end{align*}
The conclusion follows since all the factors involved have bounded absolute values.
\end{proof}

In particular, Proposition \ref{final:prop2} (i) implies that $|t_i-\xi s_i| \asymp 1$ 
for each large enough $i\in I$.  Analyzing in the same way the estimate of 
Proposition \ref{final:prop1} (ii), we find the following relation.

\begin{proposition}
\label{final:prop3}
Let $(a,b)$ be as in Lemma \ref{final:lemma0} and let $\kappa=\norm{(a,b)}$.  
For each pair of consecutive elements $g<h$ of $I$ with $h\notin J$, we have 
\begin{equation}
\label{final:prop3:eq}
 (s_g,t_g) = \pm\kappa^{-1}(a,b) + \cO(X_h^{-\sigma})
\end{equation}
where $\sigma$ is as in Proposition~ \ref{final:prop1}.  If $h$ is large enough, we also 
have $|at_h-bs_h| \asymp 1$.
\end{proposition}

\begin{proof}
We may assume that the pair $g<h$ comes form a sequence of consecutive 
elements $g<h<i<j$ of $I$ with $h\notin J$ large enough so that 
Proposition~\ref{final:prop1} applies.  Using the formula of 
Corollary~\ref{final:cor2} for $\wdeltax_g$, we find that
\begin{align*}
 X_h^{-0.3131} 
  &\gg |\det(\wddeltax_h, a\wdeltax_g^+ - b\wdeltax_g^-)| \\
  &= \big| (at_g-bs_g)\det(\wddeltax_h,(s_g,t_g))\big| + \cO(X_h^{-\sigma}).
\end{align*}
Using the formula of Corollary~\ref{final:cor2} for $\wdeltax_h$ and 
Proposition \ref{final:prop2} (i), we also note that
\[
 \big|\det(\wddeltax_h,(s_g,t_g))\big| 
  = \big| (t_h-\xi s_h) (s_gt_h-s_ht_g)\big| + \cO(X_h^{-\sigma/\gamma}) \asymp 1.
\]
So, we conclude that 
\[
  |at_g-bs_g| \ll X_h^{-\sigma}.
\]
If $|b| \le |a|$, we have $1\le |a|=\kappa \ll 1$ and this gives 
$t_g=(b/a)s_g+\cO(X_h^{-\sigma})$, thus
\[
 (s_g,t_g) = s_g \big(1, b/a \big) + \cO(X_h^{-\sigma}).
\]
Since $\norm{(s_g,t_g)} =1$, this implies that $s_g=\pm 1+ \cO(X_h^{-\sigma})$
and \eqref{final:prop3:eq} follows.  The case where $|a|\le |b|$ is similar and 
also yields \eqref{final:prop3:eq}.  Using this formula for $(s_g,t_g)$ and assuming 
$h$ large enough, Proposition \ref{final:prop2} (i) gives
\[
 |at_h-bs_h| = \kappa |s_gt_h-t_gs_h| + \cO(X_h^{-\sigma}) \asymp 1.
\qedhere
\]
\end{proof}

We deduce the following strengthening of Corollary \ref{final:cor1}.

\begin{corollary}
\label{final:cor3}
Let $\sigma$ be as in Proposition \ref{final:prop1}.  For any large enough 
pair of consecutive elements $g<h$ of $I$ with $h\notin J$, we have
\[
 \mathrm{(i)} \quad  |\det(\wdeltax_g^-,\wdeltax_g^+)| \asymp X_h^{-\sigma}
 \qquad\text{and}\qquad
 \mathrm{(ii)} \quad |\det(\wdeltax_h^-,\wdeltax_h^+)| \asymp X_h^{-\sigma/\gamma}.
\]
\end{corollary}

\begin{proof}
For large enough consecutive elements $g<h<i<j$ of $I$ with $h\notin J$, 
we have $j\notin J$ and we find
\[
 |\det(\wddeltax_j, a\wdeltax_h^+ - b\wdeltax_h^-)|
  = \big| (t_j-\xi s_j) (at_h-bs_h) (s_jt_h-s_ht_j) \big| +\cO(X_h^{-\sigma/\gamma})
 \asymp 1
\]
using the formulas of Corollary~\ref{final:cor2} and the estimates of 
Proposition~\ref{final:prop2}.  This gives estimate~(i) of the corollary as a 
consequence of Proposition~\ref{final:prop1} (iii).  Finally, estimate~(ii) follows 
from (i) with $g$ replaced by $i$, together with Proposition~\ref{final:prop1}~(i), 
similarly as in the proof of Corollary~\ref{final:cor1}~(ii).
\end{proof}

\begin{proposition}
\label{final:prop4}
Let $\sigma$ be as in Proposition \ref{final:prop1}.  For any large enough 
consecutive elements $g<h<i<j$ of $I$ with $h\notin J$, we have
\[
 \norm{C_{g,h}} \asymp X_h^{\gamma^2\lambda-1},   
 \quad
 L(C_{g,h}) \asymp X_h^{-\lambda/\gamma^2},   
 \quad
 \norm{C_{h,j}} \asymp X_h^{\gamma(3\lambda-1)},  
 \quad
 L(C_{h,j}) \asymp X_h^{\gamma^2\lambda-\gamma}.  
\]
\end{proposition}

\begin{proof}
Using Lemma \ref{final:lemma2} and the estimates of the previous corollary,
we find that
\begin{align*}
 \norm{C_{g,h}} 
   &\asymp X_h L_g^2 |\det(\wdeltax_g^-,\wdeltax_g^+)| 
     \asymp X_h^{1-2\lambda-\sigma} = X_h^{\gamma^2\lambda-1},\\ 
 \norm{C_{h,j}} 
   &\asymp X_j L_h^2 |\det(\wdeltax_h^-,\wdeltax_h^+)| 
     \asymp X_h^{\gamma-2\theta\lambda-\sigma/\gamma} 
     = X_h^{\gamma(3\lambda-1)}.
\end{align*}
By Proposition~\ref{first:prop2}~(iii), we also have
\[
 L(C_{g,h}) \asymp X_g/X_h 
  \asymp X_h^{\theta/\gamma-1}=X_h^{-\lambda/\gamma^2}.
\]
Finally, Lemma~\ref{CE:lemmaC} gives
\[
 \Delta C_{h,j} = x_{h,0}\det(\Delta^2\ux_h,\Delta^2\ux_j) + \cO(L_h^2L_j)
\]
where $x_{h,0}$ is the first coordinate of $\ux_h$.  Using the formulas of 
Corollary \ref{final:cor2} together with the estimates of 
Proposition~\ref{final:prop2}, we find that
\[
 |\det(\wddeltax_h,\wddeltax_j)|
  = \big| (t_h-\xi s_h) (t_j-\xi s_j) (s_ht_j-s_jt_h) | + \cO(X_h^{-\sigma/\gamma})
 \asymp 1,
\]
and so
\[
 L(C_{h,j}) =|\Delta C_{h,j}| \asymp X_hL_hL_j 
  \asymp X_h^{1-\theta\lambda-\gamma\theta\lambda}
  =X_h^{\gamma^2\lambda-\gamma}.
\qedhere
\]
\end{proof}

\subsection*{Final contradiction}
Let  $f<g<h<i<j<k<l$ be consecutive elements of $I$ with $h\notin J$.   
If $h$ is large enough, we have 
\[
 \{f,h,j,l\}\subset I\setminus J,
 \quad 
 \{g,i,k\}\subset J,
 \quad
 X_f\asymp X_h^{1/\gamma},
 \quad
 X_l\asymp X_h^{\gamma^2},
\]
and Proposition~\ref{final:prop4} gives
\[
\begin{array}{ll}
 \norm{C_{k,l}} \asymp X_h^{\gamma^4\lambda-\gamma^2}
   = X_h^{0.2915\dots},   
 \quad
 &L(C_{k,l}) \asymp X_h^{-\lambda}= X_h^{-0.4245\dots},  \\[5pt] 
 \norm{C_{f,h}} \asymp X_h^{3\lambda-1} = X_h^{0.2735\dots},  
 \quad
 &L(C_{f,h}) \asymp X_h^{\gamma\lambda-1} = X_h^{-0.3131\dots}.  
\end{array}
\]
Using the standard estimate \eqref{not:eq5} for determinants, we deduce that
\[
 |\det(C_{f,h},C_{k,l})|
  \ll \norm{C_{f,h}}L(C_{k,l})+\norm{C_{k,l}}L(C_{f,h})
  \ll X_h^{-0.021}.
\]
As this determinant is an integer, it vanishes if $h$ is large enough, and we conclude 
that $C_{f,h}=\rho C_{k,l}$ for some non-zero $\rho\in\bQ$ that depends on $h$.  
If $h$ is large enough, we also note that $\norm{C_{f,h}} < \norm{C_{k,l}}$ and
$L(C_{f,h}) > L(C_{k,l})$, as claimed in \eqref{final:eq1}.  This is impossible since
the first inequality implies that $|\rho|<1$ while the second yields $|\rho|>1$.
This contradiction completes the proof of Theorem \ref{intro:thm2}.

%
%

\section{Addendum}
\label{sec:add}

Although the above shows that the hypotheses of Theorem \ref{intro:thm1} are not
satisfied for $\lambda=\lambda_3$, it is nevertheless usefull to search for further 
polynomial relations satisfied by the sequence $(\ux_i)_{i\in I}$, 
assuming that $\lambda=\lambda_3$, because these relations  
may continue to hold for smaller values of $\lambda$.  They may also suggest new 
constructions that will eventually produce some $\xi\in\bR$ with $[\bQ(\xi):\bQ]>3$ 
whose exponent $\hlambda_3(\xi)$ is largest possible, in a similar way as it is done
in \cite{R2004} for the exponent $\hlambda_2(\xi)$. 

I found several such relations.  For shortness, I will 
simply indicate one of them.  It is linked with the polynomial map 
$\Xi\colon (\bR^4)^3\to\bR^4$ given by
\begin{align*}
 \Xi(\ux,\uy,\uz)
  &=C(\uz,\ux)^-\Psi_+(\uy,\ux,\uz) - C(\uz,\ux)^+\Psi_-(\uy,\ux,\uz)\\
  &=-\det(E(\ux,\uz,\uy),C(\uz,\ux))\ux -\det(C(\ux,\uz),C(\uz,\ux))\uy
      +\det(C(\ux,\uy),C(\uz,\ux))\uz.
\end{align*}
This polynomial map has algebraic properties that are similar to the map 
from $(\bR^3)^2$ to $\bR^3$ that plays a central role in 
\cite[Corollary 5.2]{R2004} and sends a pair $(\ux,\uy)$ to
$[\ux,\ux,\uy]$ in the notation of \cite[\S 2]{R2004}.  The present map 
sends $(\bZ^4)^3$ to $\bZ^4$, and it can be shown 
(or checked on a computer) that, for any $\ux,\uy,\uz\in\bR^4$, the point
$\uw=\Xi(\ux,\uy,\uz)\in\bR^4$ satisfies
\begin{align*}
 C(\uw,\ux)
      &= \det(C(\uz,\ux),C(\uz,\uy)) \det(C(\ux,\uy),C(\ux,\uz)) C(\ux,\uz),\\[2pt]
 C(\ux,\uw)
      &= \det(C(\ux,\uy),C(\ux,\uz)) C(\uz,\ux),\\[2pt]
 \Xi(\ux,\uz,\uw)
      &=\det(C(\uw,\ux),C(\ux,\uw))\,\uz \\
      &= \det(C(\uz,\ux),C(\uz,\uy)) \det(C(\ux,\uy),C(\ux,\uz))^2
           \det(C(\ux,\uz),C(\uz,\ux))\,\uz
\end{align*}
It can also be shown that, for $\ux,\uy,\uz$ as in Proposition \ref{Psi:prop1},
the point $\uw$ has
\begin{align*}
 L(\uw)
      &\ll \norm{\uz}^2L(\ux)^3L(\uy)L(\uz),\\[2pt]
 \norm{\uw}
      &\ll \norm{\uz}^2L(\ux)^3L(\uy)L(\uz) + \norm{\ux}^2\norm{\uz}L(\ux)L(\uy)L(\uz)^2.
\end{align*}

Suppose that $\lambda=\lambda_3$, and let $j_1<j_2<j_3<\dots$ denote the 
elements of $I$ in increasing order.  Without loss of generality, by dropping 
the first element of $I$ if necessary, we may assume that $j_{2i-1}\in J$ and $j_{2i}\notin J$
for each large enough $i$.  Then, upon setting $\uy_i=\ux_{j_i}$ for each $i\ge 1$, 
one finds using the above estimates that, when $i$ is large enough,
\[
 \det(\uy_{2i-6},\uy_{2i-5},\uy_{2i-4}, \Xi(\uy_{2i},\uy_{2i+1},\uy_{2i+2})) =0.
\]

\section*{Acknowledgments}
The author warmly thanks Anthony Po\"els for careful reading and useful comments.  This 
research was partially supported by an NSERC discovery grant.

\end{document}